\documentclass[12pt,reqno]{amsart}
\usepackage{amsmath,amsfonts,amssymb}
\usepackage{enumerate}
\usepackage{setspace}
\usepackage{xcolor}
\usepackage[margin=0.9in]{geometry}

\usepackage[symbol]{footmisc}

\usepackage{hyperref}

\usepackage[pagewise]{lineno}

\numberwithin{equation}{section}

\newtheoremstyle{newline}
  {3pt}
  {3pt}
  {\itshape}
  {}
  {\bfseries}
  {.}
  {.5em}
  {}
\theoremstyle{newline}\newtheorem{theorem}{Theorem}[section]
\theoremstyle{newline}
\theoremstyle{newline}
\theoremstyle{newline}\newtheorem{lemma}{Lemma}[section]
\theoremstyle{newline}
\theoremstyle{newline}\newtheorem*{remark}{Remark}
\newtheoremstyle{newlinedefn}
  {3pt}
  {3pt}
  {}
  {}
  {\bfseries}
  {.}
  {.5em}
  {}
\theoremstyle{newlinedefn}\newtheorem{definition}{Definition}[section]

\newcommand{\bu}{\boldsymbol{u}}
\newcommand{\bv}{\boldsymbol{v}}

\newcommand{\be}{\boldsymbol{e}}
\newcommand{\boldf}{\boldsymbol{f}}

\newcommand{\beps}{\boldsymbol{\varepsilon}}

\newcommand{\betab}{\boldsymbol{\beta}}
\newcommand{\gammab}{\boldsymbol{\gamma}}

\newcommand{\bbI}{\mathbb{I}}
\newcommand{\bbT}{\mathbb{T}}
\newcommand{\bbS}{\mathbb{S}}
\newcommand{\bbU}{\mathbb{U}}
\newcommand{\bbTd}{\mathbb{T}^{\delta}}

\DeclareMathOperator{\tr}{tr}

\begin{document}

 \title[Existence of weak solutions to a viscoelastic strain-limiting model]{Existence of large-data global weak solutions\\to a model of a strain-limiting viscoelastic body}
%
%
%

 \maketitle
\begin{center}
\footnotesize
 \textsc{MIROSLAV BUL\'{I}\v{C}EK\textsuperscript{1}, VICTORIA PATEL\textsuperscript{2,}\footnote[1]{Corresponding author.
 \\
 \\
}, YASEMIN \c{S}ENG\"{U}L\textsuperscript{3}, AND ENDRE S\"{U}LI\textsuperscript{2}}

\scriptsize 
 
 \textsuperscript{1}\textit{Faculty of Mathematics and Physics,
Charles University Prague, Sokolovsk\'{a} 83, 186 75 Prague 8,
Czech Republic. }

 \textsuperscript{2}\textit{Mathematical Institute, University of Oxford, Andrew Wiles Building, Woodstock Road, Oxford OX2 6GG, UK.}
 
 \textsuperscript{3}\textit{Faculty of Engineering and Natural Sciences, Sabanci University, Tuzla 34956, Istanbul, Turkey.} 

 \normalsize
\end{center}

 \begin{abstract}
 We prove the existence of a unique large-data global-in-time weak solution to a class of models of the form
$\bu_{tt} = \mbox{div }\mathbb{T} + \boldf$ for viscoelastic bodies exhibiting strain-limiting behaviour, where the constitutive equation, relating the linearised strain tensor $\beps(\bu)$ to the Cauchy stress tensor $\bbT$, is assumed to be of the form $\beps(\bu_t) + \alpha\beps(\bu) = F(\bbT)$, where we define \( F(\bbT) = ( 1 + |\bbT|^a)^{-\frac{1}{a}}\bbT\), for constant parameters $\alpha \in (0,\infty)$ and $a \in (0,\infty)$, in any number $d$ of space dimensions, with periodic boundary conditions. The Cauchy stress $\bbT$ is shown to belong to $L^{1}(Q)^{d \times d}$ over the space-time domain $Q$. In particular, in three space dimensions, if~$a \in (0,\frac{2}{7})$, then in fact $\bbT \in L^{1+\delta}(Q)^{d \times d}$ for a $\delta > 0$, the value of which depends only on $a$.
 \end{abstract}

\section{Introduction}\label{pp:secintro}
The development of a thermodynamically consistent implicit constitutive theory for continuum mechanics was initiated by Rajagopal in \cite{onICT}. While, as was noted in \cite{onICT}, implicit constitutive theories within a purely mechanical format had been studied earlier, ``these early studies do not recognize the important role thermodynamics has to play in the response of materials, nor do they provide a systematic framework that allows one to incorporate issues concerning constraints and material symmetry.'' In contrast, the theoretical framework proposed in \cite{onICT} is a systematic and thermodynamically consistent way of justifying implicit and nonlinear relationships between stress, strain and their derivatives for elastic and viscoelastic materials. In classical models of such bodies, the strain is given as a function of the stress, which leads to an explicit, usually linear, relationship between the Cauchy stress tensor \( \bbT \) and the linearised strain tensor \( \beps\) (and the derivative \( \beps_t\) in the viscoelastic case) under the small displacement gradient assumption. However, it is well-known that linear models are insufficient to describe all physical phenomena in the small strain range; see, for example, \cite{RN93}, \cite{freed2016viscoelastic} and \cite{RN105}, to name just a few, thus justifying the need for alternative models. One of the key contributions in \cite{RN20} was to use linearisations of nonlinear constitutive models to justify the presence of small strain in nonlinear models.

We are interested in a particular subclass of these implicit constitutive models, which represent a generalisation of the classical Kelvin--Voigt model to strain-limiting viscoelastic solids.
Mathematically, the term strain-limiting  refers to the fact that the strain is \textit{a priori} bounded. As we study viscoelastic solids rather than elastic ones, a linear combination of the linearised strain and its time derivative will be taken to be a bounded function of the stress.
In this paper, we will focus on the rigorous mathematical analysis of the initial-boundary-value problem that results from such models.

Strain-limiting behaviour has been exhibited in experimental procedures by certain biological matter \cite{RN97}, for example. Models of
strain-limiting materials are also of importance in the study of fracture mechanics and crack propagation. Fracture of a brittle material can occur in the small-strain range and is caused by a large stress. However, if we use a linear constitutive relation to model a body, the strain will behave like \( O(r^{-\frac{1}{2}}) \), where \( r\) is the distance to the crack tip \cite{onthenonlinearelasticresponse}. This contradicts the small displacement gradient assumption under which the model is derived. Hence, a better model might ensure that the strain is \textit{a priori} bounded, as in the case of strain-limiting models.
Another class of problems in which strain-limiting models are of relevance are those that concern bodies that undergo a small local deformation in the presence of a concentrated load. Concentration of load causes a large stress in a small subregion of the body. Thus, in order to ensure that the assumptions under which the model has been derived are not violated, one should ensure that the strain remains small.
We refer to \cite{csengul2018viscoelasticity} and the references therein for further discussion on  such models.

\subsection{Basic kinematics and derivation of the problem}
We begin by discussing how the problem of interest, stated as system (\ref{pp:equ1}) in the next section, can be deduced from standard balance equations in continuum mechanics under appropriate assumptions.

Let \( \Omega \subset\mathbb{R}^3 \) be the \it initial configuration \normalfont of a given body and \( \Omega_t \subset\mathbb{R}^3 \) the \it current configuration \normalfont of the body at time \( t\in (0, \infty)\). The initial configuration is assumed to be a stress-free state.  For a point \( \mathbf{X}\in \Omega \), we denote by \( \mathbf{x} = \chi(t, \mathbf{X}) \in \Omega_t\) the position at time $t$ of that point. The \it displacement \normalfont is given by \( \bu(t,x) := \mathbf{x} - \mathbf{X}\). The \it velocity \normalfont \( \bv\) and \it deformation gradient \normalfont \( \mathbb{F}\) are defined by
\[
\bv  := \frac{\partial \chi}{\partial t}, \quad \mathbb{F} := \frac{\partial\chi}{\partial \mathbf{X}}.
\]
The \it left Cauchy--Green stretch tensor \normalfont \( \mathbb{B}\) is defined by
\[
\mathbb{B }  := \mathbb{F}\mathbb{F}^{\mathrm{T}}.
\]
The velocity gradient \( \nabla_{\mathbf{x}}\bv \) is denoted by \( \mathbb{L}\)
with the \it linearised strain \normalfont \( \beps\) and \it symmetric part of the velocity gradient \normalfont \( \mathbb{D}\) given, respectively, by
\[
\beps = \beps(\bu):= \frac{1}{2}\big( \nabla_{\mathbf{X}}\bu + (\nabla_{\mathbf{X}}\bu)^{\mathrm{T}}\big), \quad \mathbb{D} := \frac{1}{2}\big( \mathbb{L} + \mathbb{L}^{\mathrm{T}}\big).
\]
With respect to the initial configuration, the system of balance equations for mass, linear momentum and angular momentum can be expressed in the following way:
\begin{equation}\label{pp:equ52}
\begin{aligned}
\rho &= \rho_0 \mathrm{det}(\mathbb{F}),
\\
\rho_0 \frac{\partial^2 \chi}{\partial t^2} &= \mathrm{div}_{\mathbf{X}} \mathbb{S} + \rho_0 \boldf,
\\
\mathbb{S}\mathbb{F}^{\mathrm{T}} &= \mathbb{F}\mathbb{S}^{\mathrm{T}},
\end{aligned}
\end{equation}
where \( \rho \) is the \it current density\normalfont , \( \rho_0 \) is the \it initial density\normalfont , \( \boldf\) is the \it density of body forces\normalfont , and \( \mathbb{S} \)
is the \it first Piola--Kirchhoff stress tensor\normalfont .
To close the system, a constitutive relation is required.
Following Rajagopal and Saccomandi in \cite{RN45}, we consider the following class of implicit constitutive relations:
\begin{equation}\label{pp:equ49}
\mathcal{F}( \bbT, \mathbb{B}, \mathbb{D}) = \mathbf{0}.
\end{equation}
Taking account of  representation theorems for isotropic functions, as well as the classical Kelvin--Voigt model, we consider the following subclass of constitutive relations:
\begin{equation}\label{pp:equ50}
\tilde{\nu} \mathbb{D} + \tilde{\alpha} \mathbb{B} = \tilde{\beta}_0 \bbI + \tilde{\beta}_1 \bbT + \tilde{\beta}_2 \bbT^2,
\end{equation}
where \( \tilde{\beta}_i = \tilde{\beta}_i(I_1, I_2, I_3) \) for \( i \in \{ 1,2,3\}\), are the material moduli, \( I_1 = \tr\bbT\), \( I_2 = \frac{1}{2}\tr\bbT^2\), \( I_3\ = \frac{1}{3}\tr\bbT^3\) are the principle invariants, and \( \tilde{\nu}\), \( \tilde{\alpha}\) are positive constants.

Linearising (\ref{pp:equ52}) and (\ref{pp:equ50}) under the assumption that
\begin{equation}\label{pp:equ51}
\max_{\mathbf{X}\in \Omega,\, t\in [0, \infty)} |\nabla_{\mathbf{X}}\bu |  = O(\delta),
\end{equation}
for some \( \delta \ll 1\), the constitutive relation reduces to
\begin{equation}\label{pp:equ55}
\nu \beps(\bu_t) + \alpha\beps(\bu) = \beta_0 \bbI + \beta_1 \bbT + \beta_2 \bbT^2,
\end{equation}
where \( \beta_i = \beta_i(I_1, I_2, I_3) \), \( i \in \{ 1,2,3\}\), and \( \nu
\), \( \alpha\) are positive constants. We use \( \beps\) to denote the symmetric gradient operator \( \frac{1}{2}( (\nabla\,\cdot\,) + (\nabla\,\cdot\,)^{\mathrm{T}} ) \) with derivatives  taken with respect to \( \mathbf{X}\), 
and for $\mathbb{A} \in \mathbb{R}^{d\times d}$ we shall denote by \( |\mathbb{A}|:=(\mathbb{A}:\mathbb{A})^{\frac{1}{2}}\) the Frobenius norm of \(\mathbb{A}\). The same notation will be used to denote the Euclidean norm of a $d$-component vector and the absolute value of a real number; it will be clear from the context which of these interpretations of the symbol $|\cdot|$
is intended.
We note that under the assumption (\ref{pp:equ51}), derivatives with respect to \( \mathbf{X}\) and \( \mathbf{x}\) may be used interchangeably. In particular, we have that \( \beps(\bu_t ) \approx\mathbb{D}\).

The balance equations (\ref{pp:equ52}) thereby reduce  to
\begin{equation}\label{pp:equ56}
\begin{aligned}
\rho_0 &= \rho ( 1 + \tr\beps),
\\
\rho_0 \bu_{tt} &= \mathrm{div}_{\mathbf{X}} \bbT + \rho_0 \boldf,
\\
\bbT &= \bbT^{\mathrm{T}},
\end{aligned}
\end{equation}
up to an error of order \( \delta^2 \). For the sake of simplicity we will assume henceforth that $\rho_0 \equiv 1$ and \( \nu = 1\). If \( \nu = 0 \), problem (\ref{pp:equ56}), (\ref{pp:equ55}) models a strain-limiting elastic body.  The form of (\ref{pp:equ55}) that we are particularly interested in is
\begin{equation}\label{pp:equ58}
\beps(\bu_t)  + \alpha\beps(\bu) = \frac{\bbT}{(1 + |\bbT|^a)^{\frac{1}{a}}} =: F(\bbT),
\end{equation}
for a positive parameter \( a\in (0, \infty) \), corresponding to $\beta_0 = \beta_2=0$, and $\beta_1 = (1+|\bbT|^a)^{-\frac{1}{a}}$. As \( a\) is treated as a fixed parameter we do not explicitly indicate the dependence of 
\( F\) on \( a\).
Furthermore, we note that \( F\) is a bounded function on \( \mathbb{R}^{d\times d}\), encapsulating the strain-limiting property. Indeed, it follows that the linearised strain \( \beps(\bu ) \) is bounded by the following reasoning. Multiplying both sides of (\ref{pp:equ58}) by \( \mathrm{e}^{\alpha t}\) and integrating with respect to the time variable, we deduce that
\begin{align*}
|\beps(\bu(t))| = \bigg|\mathrm{e}^{-\alpha t} \beps(\bu_0)  + \int_0^t \mathrm{e}^{\alpha(s-t)} F(\bbT(s)) \,\mathrm{d}s\bigg|\leq \mathrm{e}^{-\alpha t}\|\beps(\bu_0) \|_\infty + \frac{1}{\alpha}(1-\mathrm{e}^{-\alpha t}),\qquad t>0,
\end{align*}
where \(\bu_0 = \bu(0) \). In particular, we see that if \( \|\beps(\bu_0 ) \|_\infty\leq \alpha^{-1}\), then \( \|\beps(\bu(t) )\|_\infty\leq \alpha^{-1}\) for all \( t\in (0, \infty) \). We note that this particular choice of \( F\) has already been studied in the one-dimensional case, for example, in \cite{RN14} and \cite{RN121}.

For the sake of completeness we shall discuss the thermodynamic motivation for this choice of the function \( F\). In particular, we will show that the total energy associated with the model is decreasing in time, provided that the body force \( \boldf\equiv \mathbf{0}\). In order to simplify the presentation of the following analysis, we divide (\ref{pp:equ58}) through by \( \alpha\).  Thus we are motivated to define the function \( f_\alpha: \mathbb{R}^{d\times d}\rightarrow [0, \infty) \) by
\begin{align*}
f_\alpha(\bbT) := \frac{1}{\alpha}\int_0^{|\bbT|} \frac{t}{( 1+ t^a)^{\frac{1}{a}}}\,\mathrm{d}t,
\end{align*}
for \( \bbT \in \mathbb{R}^{d\times d}\). In particular, we have that \( \frac{\partial f_\alpha}{\partial\bbT}(\bbT) = \frac{1}{\alpha}F(\bbT) =: F_\alpha(\bbT)\).  By direct calculation, \( \nabla^2_{\bbT}f_\alpha\) is positive definite and continuous. Thus we see that \( f_\alpha\) is a convex \( C^2\) function on \( \mathbb{R}^{d\times d}\). We define the convex conjugate \( f_\alpha^*\) on \( \mathbb{R}^{d\times d}\) by
\[
f^*_\alpha (\bbS) = \sup_{\bbT \in \mathbb{R}^{d\times d}} \big( \,\bbS : \bbT - f_\alpha(\bbT) \,\big).
\]
By the Fenchel--Young inequality and the differentiability properties of \( f_\alpha\), we have that the supremum is either infinity or it is attained when $\bbS = F_\alpha(\bbT)$, and therefore
\[
f_\alpha(\bbT) + f_\alpha^*(F_\alpha(\bbT))  =  F_\alpha(\bbT) : \bbT, \qquad \bbT \in \mathbb{R}^{d \times d}.
\]
Noting that \( F_\alpha^{-1}\) is well-defined on the set of matrices with Frobenius norm strictly less than \(\alpha^{-1}\), we also have that
\[
f_\alpha(F^{-1}_\alpha(\beps)) + f^{*}_\alpha(\beps) = \beps : F_\alpha^{-1}(\beps),
\]
for every \( \beps\in \mathbb{R}^{d\times d}\), \( |\beps|< \alpha^{-1}\). With this in mind, we can write
\begin{align*}
\bbT : \beps(\bu_t) &= \Big( \bbT - \frac{\partial f^*_\alpha}{\partial \beps}(\beps(\bu))\Big) \beps(\bu_t) + \frac{\partial f^*_\alpha}{\partial \beps}(\beps(\bu)) : \beps(\bu_t)
\\
&= \Big( \bbT - F_\alpha^{-1}(\beps(\bu))\Big): \Big( F(\bbT) - \alpha\beps(\bu)\Big) + \frac{\partial}{\partial t}\Big( f_\alpha^*(\beps(\bu))\Big)
\\
&= \alpha\Big( \bbT - \bbT_{0,\alpha}\Big): \Big( F_\alpha(\bbT) - F_\alpha(\bbT_{0,\alpha})\Big)+ \frac{\partial}{\partial t}\Big( f_\alpha^*(\beps(\bu))\Big),
\end{align*}
where \( \bbT_{0,\alpha}\) is the unique element of \( \mathbb{R}^{d\times d}\) with \( \bbT_{0,\alpha} = F_\alpha^{-1}(\beps(\bu))\). Using (\ref{pp:equ56})\textsubscript{2} with $\rho_0 \equiv 1$ and assuming that the boundary conditions are chosen so that when integrating by parts all boundary integrals vanish, we get
\begin{align*}
0&= \int_\Omega \bu_{tt}\cdot \bu_t + \bbT : \beps(\bu_t) \,\mathrm{d}x
\\
&= \int_\Omega \frac{\partial}{\partial t}\bigg( \frac{|\bu_t|^2}{2} + f_\alpha^*(\beps(\bu)) \bigg) + \alpha (\bbT - \bbT_{0,\alpha}) : (F_\alpha(\bbT) - F_\alpha(\bbT_{0,\alpha})) \,\mathrm{d}x
\\
&= \int_\Omega \frac{\partial}{\partial t}\bigg( \frac{|\bu_t|^2}{2} + f_\alpha^*(\beps(\bu)) \bigg) +(\bbT - \bbT_{0,\alpha}) : (F(\bbT) - F(\bbT_{0,\alpha})) \,\mathrm{d}x.
\end{align*}
This yields
\begin{equation}\label{pp:equ59}
\begin{aligned}
&\int_\Omega\frac{|\bu_t(t) |^2}{2} + f_\alpha^*(\beps(\bu(t)))\,\mathrm{d}x + \int_0^t\int_\Omega (\bbT - \bbT_{0,\alpha}) : (F(\bbT) - F(\bbT_{0,\alpha})) \,\mathrm{d}x\,\mathrm{d}s
\\&\quad\quad
= \int_\Omega \frac{|\bu_t(0)|^2}{2} + f_\alpha^*(\beps(\bu(0)))\,\mathrm{d}x,
\end{aligned}
\end{equation}
for every \( t\in (0, \infty) \). The first term on the left-hand side represents the total energy (i.e., the sum of the kinetic and the stored energy) and the second term is a nonnegative, increasing function of \( t\) by the monotonicity of \( F\). Thus the total energy is a decreasing function of \( t\). Furthermore, we note that
\[
\frac{\partial f_\alpha^*}{\partial \bbT}(\bbT) = F^{-1}_\alpha(\bbT) = \frac{\alpha\bbT}{( 1 - \alpha^a |\bbT|^a)^{\frac{1}{a}}}.
\]
In particular, we have \( f_\alpha^*(\bbT) = \infty \) if \( |\beps(\bu) |> \alpha^{-1}\). (See \cite{RN6} and references therein for details of the reasoning from convex analysis.) Thus we must have \(|\beps(\bu) |\leq \alpha^{-1}\) a.e. in \(  (0, \infty) \times \Omega \), provided that the right-hand of (\ref{pp:equ59}) is finite.

\subsection{Statement of the model problem.} We will make some mathematical simplifications in order to make the rigorous mathematical analysis of the problem more manageable.
Neglecting (\ref{pp:equ56})\textsubscript{1} and assuming that \( \rho_0 \equiv 1\), we obtain the following system of equations:
\begin{equation}\label{pp:equ53}
\begin{aligned}
\bu_{tt} &= \mathrm{div}_{\mathbf{X}}\bbT + \boldf,
\\
\beps(\bu_t) + \alpha\beps(\bu) &= F(\bbT),
\\
\bbT&= \bbT^{\mathrm{T}},
\end{aligned}
\end{equation}
in \( [0, \infty) \times \Omega \), with suitable initial and boundary conditions.
We note that if, more generally, \( \rho_0 \in W^{1,\infty}(\Omega;\mathbb{R}_{\geq 0})\) is such that \(\rho_0 \) is uniformly bounded away from \( 0 \), then it is not much more difficult to include a variable density.

A further simplification that we will  make is to take \( \Omega = (0, 1)^d\), with dimension \( d\geq 2\), and to supplement the problem with a periodic boundary condition. This allows us to work with Fourier basis functions in a finite-dimensional approximation of the problem, leading to a sequence of semi-discrete numerical approximations. Because of this simplification, we can view the work here as a time-dependent analogue of the problem studied in \cite{RN8}, dealing with viscoelastic solids rather than elastic ones.
We note that despite its geometrical simplicity, one can still use the framework presented here to study the effects of concentrated loads that are active in the neighbourhood of the centre point of the periodic cell, assuming that the side-length of the cell is large enough so that the effects of concentrations are not effective in the neighbourhood of the boundary of the cell. Trivially, the results of the paper extend to any axiparallel parallelepiped $\Omega \subset \mathbb{R}^d$, $d \geq 2$, independent of the edge-lengths of $\Omega$.

The form of \( F\) studied in this paper  has already been considered in the purely elastic case, for example, in \cite{RN8}. For elastic solids, it has been shown in \cite{RN8} that a unique weak solution to the static problem in the periodic setting exists for every \( a\in (0, \frac{2}{d}) \) and a renormalised solution exists for every $a \in (0,\infty)$, with coincidence of these two notions of solution for the range $a \in (0,\frac{2}{d})$; and in the case of a homogeneous Dirichlet boundary condition the existence of a unique weak solution was shown for $a \in (0,\frac{1}{d})$ in \cite{RN9}. When considering mixed Dirichlet--Neumann boundary conditions, the authors of \cite{RN6} were able to show existence of a weak solution up to a penalisation on the Neumann part of the boundary. The extension of this work to the time-dependent problem, as well as further open problems, will be discussed at the end of the paper.

The constitutive relation proposed by Rajagopal in \cite{RN19} and \cite{RN20} is in fact
\begin{equation}\label{pp:equ54}
F(\bbT) = f_0(\tr\bbT, \tr\bbT^2 )\, \bbI + \frac{\bbT}{\mu_0(1 + |\bbT|^a)^{\frac{1}{a}}}.
\end{equation}
However, we note that under suitable, physically reasonable, structural assumptions on the real-valued-function \( f_0 \) the extension of the results in this paper to a problem with a constitutive relation given by (\ref{pp:equ54}) is fairly straightforward. Hence, for the sake of simplicity of the exposition, we shall assume in what follows that \( f_0 \equiv 0 \) and \( \mu_0 = 1 \) and focus on the key difficulty in the analysis: dealing with the second summand on the right-hand side of \eqref{pp:equ54}.

A similar problem to (\ref{pp:equ53}) is discussed in \cite{ItouHiromichi2018Otso} and \cite{RN46}; the authors considered a quasi-static problem in both papers, by which we mean that the term \( \bu_{tt}\) was omitted from (\ref{pp:equ53}). The function \( F\) there was given by (\ref{pp:equ54}) with \( f_0 \equiv 0 \).  It was shown that a solution to the problem exists, using an elliptic regularisation technique similar to the one employed in this paper, combined with a fixed-point argument.
Appropriate bounds were then found in order to allow the regularisation parameter to go to \( 0 \).
A marked difference here compared to the analysis in those  papers is that the stress \( \bbT \) was only shown there to belong to the space \( C([0, T]; \mathcal{M}(\overline{\Omega})^{d\times d}) \), where \( \mathcal{M}(\overline{\Omega}) \) is the space of Radon measures on \( \overline{\Omega}\).
In contrast, we shall be able to ensure that the stress is at least in the space \( L^1(0, T; L^1(\Omega)^{d\times d}) \) in any number of space dimensions $d$; and, in three space dimensions, if $a \in (0,\frac{2}{7})$, then $\bbT \in L^{1+\delta}(0,T; L^{1+\delta}(\Omega))$ for a $\delta > 0$, the value of which depends only on $a$.
Furthermore, we are able to show that the constitutive relation holds in a pointwise sense rather than in a variational sense, as was the case in \cite{ItouHiromichi2018Otso} and \cite{RN46}. This then enables us to prove the validity of the constitutive relation by making use of the fact that the stress tensor is an integrable function.

There are some important differences between the multi-dimensional problem (\ref{pp:equ1}) and its one-dimensional counterpart. In one space dimension the symmetric gradient reduces to simply the first derivative in the spatial variable. The one-dimensional time-dependent problem has been studied in \cite{RN14} in the context of travelling wave solutions. A substitution was used in order to reduce the problem in two dependent variables to a problem in a single dependent variable.
This work was expanded upon in \cite{RN121}, where the existence of a strong solution to the one-dimensional time-dependent problem was proved on the domain \( \Omega = \mathbb{R}\). The proof relied on a substitution argument that is specific to the one-dimensional case. Although the regularity of the solutions is much stronger in \cite{RN121}, the authors were only able to prove local-in-time existence. Here, we shall prove global-in-time existence of weak solutions in any number of space dimensions.

The paper \cite{RN56}, with corresponding mathematical analysis contained in \cite{RN66} and \cite{RN48}, deals with a different generalisation of the Kelvin--Voigt model, where the stress is decomposed into an elastic component \( \bbT_e \) and a dissipative component \( \bbT_f \) associated with a viscous fluid. Linearising under the assumption (\ref{pp:equ51}), the constitutive relations are \( \bbT_e = h_1 (\beps) \) and \( \bbT_f = h_2 (\beps_t) \). However, a key point is that the functions \( h_1 \) and \( h_2 \) are not assumed to be bounded. In particular, the problem is not strain-limiting, and some of the key technical difficulties that we encounter here for the strain-limiting problem (\ref{pp:equ1}) are therefore absent.

We are now ready to proceed with our analysis of the strain-limiting problem (\ref{pp:equ1}). The paper is structured as follows. Section \ref{pp:secformulation} introduces the mathematical problem alongside the definition of a weak solution and the statement of the regularised problem, which forms the basis of the weak compactness argument that is at the heart of our proof of the existence of global-in-time large-data weak solutions. Furthermore, some useful auxiliary results that will be used in the proof are given. In Section \ref{pp:secregexist}, we focus on proving the existence of a solution to the regularised problem (\ref{pp:equ2}) by a Galerkin method using a basis of trigonometric polynomials.  We also prove various bounds on the sequence of solutions that are independent of the regularisation parameter. Certain bounds may only be proved under stronger restrictions on the data. In Section \ref{pp:seclimitn} we take the limit in the regularisation parameter and show that the accumulation point of the sequence of solutions is in fact the unique weak solution of (\ref{pp:equ1}).  A key difficulty is showing that the sequence of stress tensors converges pointwise a.e. and then using this to prove that the weak form of the PDE does indeed hold.  Then, under further restrictions on the dimension \( d\) and the parameter \( a\), we show that a stronger convergence result can in fact be proved for the sequence of stress tensors.  The main difficulty is showing that the sequence of approximate stress tensors is bounded in a reflexive Lebesgue space, from which we are able to deduce strong convergence in \( L^1((0, T) \times \Omega)\). Finally, in Section \ref{pp:secconclusion}, we discuss future work and related open problems.

\section{Formulation of problem and auxiliary results}\label{pp:secformulation}
Let \( \Omega := (0, 1)^d\subset\mathbb{R}^d \), \( d\geq 2\).
For a fixed final time \( T > 0 \), we define the time-space cylinder \( Q:= (0, T) \times \Omega\).
Given a  space \( \mathcal{F}\) of real-valued functions on \( \mathbb{R}^d \), let \( \mathcal{F}_\#\) denote the subspace of functions \( f\in \mathcal{F}\) that are \( 1\)-periodic, i.e., periodic with respect to each of the \( d\) co-ordinate directions. Let \( \mathcal{F}_*\) be the subspace of functions from \(  \mathcal{F}_\# \) whose integral over \( \Omega\) is equal to \( 0 \). For example, \( L^p_\#(\Omega) \) consists of \(1\)-periodic functions \(f\) such that \( |f|^p \) is integrable over \( \Omega\).
Let \( \|\cdot\|_p\) denote the usual norm on \( L^p(\Omega)\) for \( p \in [1,\infty ]\). When considering the norm in \( L^p(\Omega_0) \) where \( \Omega_0 \neq \Omega\), we will state it explicitly.
Furthermore, \( \mathcal{F}^d\) denotes the space of \( d\)-component vector functions such that each component is an element of \( \mathcal{F}\). The space \( \mathcal{F}^{d\times d}\) is defined analogously.
If \( \|\cdot\|_{\mathcal{F}}\) is the usual norm on \( \mathcal{F}\), the norms on \( \mathcal{F}^d\) and \( \mathcal{F}^{d\times d}\) will be given by
\(\|\cdot\|_{\mathcal{F}}:= \||\cdot |\|_{\mathcal{F}}\) where \( |\cdot| \) denotes the  absolute value on $\mathbb{R}$, the Euclidean norm on \( \mathbb{R}^d \), or the Frobenius norm on \( \mathbb{R}^{d\times d}\), as the case may be.
Let \( C^\infty_\#(\overline{\Omega})\) denote the set of smooth, real-valued functions on \( \mathbb{R}^d\) that are \( 1\)-periodic and \( C^\infty_*(\overline{\Omega}) \) the subspace of \( C^\infty_\#(\overline{\Omega})\) consisting of functions that have integral over \( \Omega\)  equal to \( 0 \). For \( p \in [1,\infty) \) and \( k \in \mathbb{N}\),  we define the Sobolev space \( W^{k,p}_\#(\Omega) \) to be the closure of \( C^\infty_\#(\overline{\Omega}) \) with respect to the norm
\begin{equation}\nonumber
\|f\|_{k,p} = \|f\|_{W^{k,p}(\Omega)} := \bigg( \sum_{i = 0}^k \sum_{|\alpha| = k} \|\partial^\alpha f\|_p^p \bigg)^{\frac{1}{p}},
\end{equation}
where \( \alpha\) is taken from the set of multi-indices in \( \mathbb{N}_0^d\).
The space \( W^{k,p}_*(\Omega) \) is defined analogously. Let \( W^{k,\infty}(\Omega)\) denote the set of functions \( f\in W^{k,1}(\Omega) \) such that \( \partial^\alpha f\in L^\infty(\Omega)\) for every \( \alpha\in \mathbb{N}_0^d\) such that \( |\alpha|\leq k \).
For a Banach space \( X\), we let \( X^\prime\) denote the dual space of \( X\) with duality pairing \(\langle\cdot,\cdot\rangle\).
Furthermore, we will use Einstein's summation convention throughout.

Given a vector field \( \boldf: Q\rightarrow\mathbb{R}^d\), initial data \( \bu_0 \), \( \bv_0:\Omega\rightarrow\mathbb{R}^d\), and model parameters $\alpha>0$, $a>0$, we seek a unique couple \( (\bu, \bbT): Q\rightarrow\mathbb{R}^d \times \mathbb{R}^{d\times d}\) such that
\begin{equation}\label{pp:equ1}
\begin{aligned}
\bu_{tt}&= \mathrm{div}(\bbT) + \boldf, &\quad &\text{ in }Q,
 \\
 \beps(\bu_t + \alpha\bu) &= \frac{\bbT}{(1 + |\bbT|^a)^{\frac{1}{a}}} =: F(\bbT), &\quad &\text{ in }Q,
 \\
 \bu(0, \cdot) &= \bu_0,  &\quad &\text{ in }\Omega,
 \\
 \bu_t(0, \cdot) &= \bv_0, &\quad &\text{ in }\Omega.
\end{aligned}
\end{equation}
We wish to prove the existence of a unique weak solution to (\ref{pp:equ1}) in the following sense. The choice of function spaces is the natural choice according to bounds that we derive in the proof of existence. For the sake of simplicity, from now on we will use the superscript \( \dot{}\) rather than the subscript~\( t\) in order to denote differentiation with respect to the time variable.

\begin{definition}\label{pp:def1}
Let \( \bu_0 \), \( \bv_0 \in L^2_*(\Omega)^d\) and \( \boldf\in L^2(0, T; L^2_*(\Omega)^d)\).
The couple \( (\bu, \bbT) \) is a weak solution of the strain-limiting problem (\ref{pp:equ1}) if
\begin{itemize}
\item \( \bu\), \( \dot{\bu}\in L^\infty(0, T; L^2_*(\Omega)^d) \),
\item \( \dot{\bu} + \alpha\bu \in L^p(0, T; W^{1,p}_*(\Omega)^d) \) for every \( p \in [1,\infty) \) with \( \beps(\dot{\bu}+\alpha\bu) \in L^\infty(Q)^{d\times d}\),
\item \( \ddot{\bu}\in L^2(0, T; L^2_*(\Omega)^d) \),
\item \( \bbT \in L^1(0, T; L^1_\#(\Omega)^{d\times d}) \),
\end{itemize}
and, for every \( \bv \in W^{1,2}_*(\Omega)^d \) such that \( \beps(\bv) \in L^\infty(\Omega)^{d\times d}\),
\begin{equation}\label{pp:equ3}
\int_\Omega\ddot{\bu}(t) \cdot \bv + \bbT(t) : \beps(\bv) \,\mathrm{d}x = \int_\Omega\boldf(t) \cdot \bv\,\mathrm{d}x,
\end{equation}
for a.e. \( t\in (0, T) \), where
\begin{equation}\label{pp:equ4}
\beps(\dot{\bu}+\alpha\bu) = F(\bbT) \quad\text{a.e. in }Q.
\end{equation}
Furthermore,  the initial conditions must hold in the following sense:
\begin{equation}\label{pp:equ5}
\lim_{t\rightarrow 0+ }\bigg( \|\bu(t) - \bu_0 \|_2 + \|\dot{\bu}(t) - \bv_0 \|_2 \bigg) = 0.
\end{equation}
\end{definition}
We note that \( \bu\), \( \dot{\bu}\in C([0, T]; L^2_*(\Omega)^d) \) by the regularity assumptions of Definition \ref{pp:def1}. (See \cite[Chapter~5]{evansPDE}, for example.) Thus (\ref{pp:equ5}) is well-defined.

The map \( F\) is a continuous, bounded and injective function.
We will also require the following properties of \( F\). The proof of these results can be found in \cite{RN8}.

\begin{lemma}\label{pp:lem1}
For any \( y \geq 0 \) and \( a> 0 \),
\begin{align*}
\min\{ 1,2^{-1+\frac{1}{a}}\} (1 + y) \leq (1 + y^a)^{\frac{1}{a}}\leq \max\{ 1,2^{-1+\frac{1}{a}}\} (1 + y).
\end{align*}
\end{lemma}

\begin{lemma}\label{pp:lem2}
Let \( a> 0 \). For any  \( \bbT \), \( \bbS \in \mathbb{R}^{d\times d}\),
\begin{align*}
(\bbT - \bbS):(F(\bbT) - F(\bbS)) \geq \max\{ 1, 2^{\frac{1}{a} - a}\}\cdot \frac{|\bbT - \bbS|^2}{(1 + |\bbT| + |\bbS|)^{1+a}}.
\end{align*}
\end{lemma}

From Lemma \ref{pp:lem2}, we deduce that \( F\) is a monotonic function.
However, it is not a bijection from \( \mathbb{R}^{d\times d}\) onto \( \mathbb{R}^{d\times d}\) since it is bounded. In order to consider finite-dimensional approximations, we would like \( F\) to be invertible on the whole of \( \mathbb{R}^{d\times d}\). Thus, in the spirit of \cite{RN9} and \cite{RN6}, for every \( n\in \mathbb{N}\) we  first consider the following regularised problem:
\begin{equation}\label{pp:equ2}
\begin{aligned}
\ddot{\bu}&= \mathrm{div}(\bbT) + \boldf, &\quad &\text{ in }Q,
 \\
 \beps(\dot{\bu} + \alpha\bu) &= \frac{\bbT}{(1 + |\bbT|^a)^{\frac{1}{a}}} + \frac{\bbT}{n( 1 + |\bbT|^{1-\frac{1}{n}})}=: F_n(\bbT), &\quad &\text{ in }Q,
 \\
 \bu(0, \cdot) &= \bu_0,  &\quad &\text{ in }\Omega,
 \\
 \dot{\bu}(0, \cdot) &= \bv_0, &\quad &\text{ in }\Omega.
\end{aligned}
\end{equation}
For the sake of simplicity, for the moment at least, we shall not indicate the dependence of \( \bu \) and \( \bbT \) on \( n\). However, we shall explicitly indicate the dependence of these functions on \( n \) in Section \ref{pp:seclimitn} where we let \( n\rightarrow\infty\).
To see that \( F_n \) is a bijection from \(\mathbb{R}^{d\times d}\) to itself, we use the Browder--Minty theorem noting the monotonicity result from Lemma \ref{pp:lem2}. Thus (\ref{pp:equ2})\textsubscript{2} is equivalent to
\[
\bbT = F_n^{-1}(\beps(\dot{\bu} + \alpha\bu) ), \quad \text{ in }Q.
\]
In analogy with Definition \ref{pp:def1}, we define a weak solution of (\ref{pp:equ2}) as follows.

\begin{definition}\label{pp:def2}
For a given \( n\in \mathbb{N}\) and \( \bu_0\), \( \bv_0 \), \( \boldf\) as in Definition \ref{pp:def2},  the couple \( (\bu, \bbT) \) is a weak solution of (\ref{pp:equ2}) if
\begin{itemize}
\item \( \bu\), \( \dot{\bu}\in L^\infty(0, T; L^2_*(\Omega)^d) \),
\item \( \dot{\bu}+\alpha\bu\in L^{n+1}(0, T; W^{1,n+1}_*(\Omega)^d) \) with \( \beps(\dot{\bu}+\alpha\bu) \in L^{n+1}(0, T; L^{n+1}_\#(\Omega)^{d\times d}) \),
\item \( \ddot{\bu}\in L^2(0, T; L^2_*(\Omega)^d) \),
\item \( \bbT \in L^{1 + \frac{1}{n}}(0, T; L^{1 + \frac{1}{n}}_\#(\Omega)^{d\times d}) \),
\end{itemize}
and, for every \( \bv\in W^{1,n+1}_*(\Omega)^d\),
\begin{equation}\label{pp:equ21}
\int_\Omega \ddot{\bu}(t) \cdot \bv + \bbT(t) : \beps(\bv) \,\mathrm{d}x = \int_\Omega \boldf(t) \cdot \bv\,\mathrm{d}x,
\end{equation}
for a.e. \( t\in (0, T) \) with
\begin{equation}\label{pp:equ22}
\beps(\dot{\bu}+\alpha\bu) = F_n(\bbT) \quad \text{a.e. in  }Q.
\end{equation}
Furthermore, the initial conditions must hold in the following sense:
\begin{equation}\label{pp:equ23}
\lim_{t\rightarrow 0+ }\bigg( \|\bu(t) - \bu_0 \|_2 + \|\dot{\bu}(t) - \bv_0 \|_2\bigg) = 0.
\end{equation}
\end{definition}

In order to prove the existence of a weak solution of (\ref{pp:equ1}), we begin by proving the existence of a weak solution of (\ref{pp:equ2}). We will require the following auxiliary results. For details of the proofs of Lemmas \ref{pp:lem3} and \ref{pp:lem4}, we refer to \cite{RN8}.

\begin{lemma}\label{pp:lem3}
Let \( \bv \in W^{1,2}_*(\Omega)^d\) be such that \( \beps(\bv) \in L^\infty_\#(\Omega)^d\). Then there exists an approximating sequence \( (\bv^n)_n \subset C^\infty_*(\overline{\Omega})^d\) such that
\begin{itemize}
\item \( \bv^n \rightarrow\bv \) strongly in \( L^2_*(\Omega)^d\), and
\item \( \beps({\bv}^n) \overset{\ast}{\rightharpoonup} \beps(\bv)\) weakly-* in \( L^\infty_\#(\Omega)^{d\times d}\).
\end{itemize}
\end{lemma}

\begin{lemma}[Korn's inequality in \( L^p\)]\label{pp:lem4}
Let \( p \in (1,\infty) \), \( d\geq 2\) and \(\Omega = (0, 1)^d\). There exists a positive constant \( c_p\) such that
\[
\| \bv\|_{1,p} \leq c_p \|\beps(\bv) \|_p \quad \forall \bv\in W^{1,p}_*(\Omega)^d.
\]
\end{lemma}

The following result is from \cite{RN114} and is needed to show that the function \( F_n \) defined in the next section is a \( C^1\)-diffeomorphism.

\begin{lemma}\label{pp:lem6}
A \( C^1\)-map \( f:\mathbb{R}^l\rightarrow\mathbb{R}^l \) is a \( C^1\)-diffeomorphism if and only if the Jacobian \( \mathrm{det}(Df)\) never vanishes and \( |f(\bv) |\rightarrow\infty \) whenever \( |\bv|\rightarrow\infty \).
\end{lemma}

\section{Existence of a solution to the approximate problem}\label{pp:secregexist}
First we will show that there exists a unique weak solution of (\ref{pp:equ2}). From this proof, we will also obtain \( n \)-independent bounds on the solution. These will be used when we consider the limit as \( n\rightarrow\infty \).

\begin{theorem}\label{pp:thm1}
Let \( n \in\mathbb{N}\), \( \alpha > 0 \) and $a>0$. Suppose that \( \bu_0\), \( \bv_0 \in L^2_*(\Omega)^d\) are such that \( \bv_0 + \alpha\bu_0 \in W^{k+1,2}_*(\Omega)^d \) for some \( k > \frac{d}{2}\) with
\begin{equation}\label{pp:equ17}
\|\beps(\bv_0 + \alpha\bu_0 ) \|_\infty \leq C_* < 1,
\end{equation}
for a constant \( C_*\in (0, 1) \). Furthermore, let \( \boldf\in L^2(0, T; L^2_*(\Omega)^{d}) \) be given.
Then, there exists a unique weak solution \((\bu, \bbT) \) of the regularised problem (\ref{pp:equ2}) in the sense of Definition \ref{pp:def2}. In addition, the following bound holds:
\begin{equation}\label{pp:equ27}
\begin{aligned}
&\sup_{t\in [0, T]}\|\bu(t) \|_2  + \sup_{t\in [0, T]}\|\dot{\bu}(t) \|_2 + \|\beps(\dot{\bu}+\alpha\bu) \|_{L^{n+1}(Q)} + \int_Q |\bbT|\,\mathrm{d}x\,\mathrm{d}t
\\
&\quad\quad
+ \sup_{t\in [0, T]}\bigg( \int_\Omega |\bbT(t) |^{1-a}\chi_{\{|\bbT(t)|\geq 1\}}\bigg) \leq C,
\end{aligned}
\end{equation}
where \( C\) is a positive constant that is independent of \(n\) and \( \chi_A\) is the indicator function of any measurable set \( A\).
\end{theorem}

The regularity requirements on the initial data are higher than one might expect. It is for a technical reason that we  demand  \( \bv_0 + \alpha\bu_0 \in W^{k+1,2}_*(\Omega)^d \) for a \( k > \frac{d}{2}\). In particular, by the Sobolev embedding theorem \(  W^{k+1,2}_*(\Omega)\) is continuously embedded into \( W^{1,\infty}_*(\Omega) \) . This will allow us to deduce the strong convergence in \( W^{1,\infty}_*(\Omega) \) of a certain sequence of approximations of \( \bv_0 + \alpha\bu_0 \).  However, the demand for \( \beps(\bv_0 + \alpha\bu_0 ) \) to be bounded in \( L^\infty(\Omega)^{d\times d}\) is natural since we eventually want to take the limit as \( n \rightarrow\infty \) and thus we would like \( F^{-1}(\beps(\bv_0 + \alpha\bu_0 ) ) \) to be well-defined.
The weakening of the conditions on \( \bv_0 + \alpha\bu_0 \) will be discussed in the first remark in Section \ref{pp:seclimitn}.

\begin{proof}

We will construct a weak solution of (\ref{pp:equ2}) by use of the Galerkin approximation method. Let \( (\phi_i)_{i=1}^\infty\) be a sequence of trigonometric polynomials from \( C^\infty_*(\overline{\Omega}) \) such that they form an orthonormal basis of \( L^2_*(\Omega) \) and for every \( m\in\mathbb{N}\) there exists an \( M_m\in \mathbb{N}\) such that the linear span of \( (\phi_i)_{i=1}^{M_m} \) is  the vector space of trigonometric polynomials of degree at most \(m \) with integral over \( \Omega\) equal to \( 0 \). Let \( V_m = (\mathrm{span}\{ \phi_1,...,\phi_{M_m}\})^d\). We note that a basis of \( V_m \) that is orthogonal with respect to the inner product in \( L^2(\Omega)^d\) is \( (\phi_i\be_j )_{i,j=1}^{M_, d}\) where \( \be_j \) is the \(j\)-th standard basis vector in \( \mathbb{R}^d\).

For each \(m \in \mathbb{N}\), we want to find a function \( \bu^m\in W^{2,2}([0, T]; L^2_*(\Omega)^d) \) of the form
\[
\bu^m(t,x) = \sum_{i=1}^{M_m}\sum_{j=1}^d \beta^m_{i,j}(t) \phi_i(x)\be_j,
\]
such that, for every \( k \in \{ 1,\dots, M_m\} \) and \( l\in \{ 1,\dots , d\} \), we have
\begin{equation}\label{pp:equ6}
\int_\Omega \ddot{\bu}^m(t) \cdot (\phi_k\be_l ) + \bbT^m(t):\beps(\phi_k\be_l) \,\mathrm{d}x = \int_\Omega \boldf(t) \cdot (\phi_k\be_l) \,\mathrm{d}x,
\end{equation}
for every \( t\in (0, T) \) with \( \bbT^m\) is defined by
\begin{equation}\label{pp:equ13}
\beps(\dot{\bu}^m + \alpha\bu^m) = F_n(\bbT^m) \quad\text{ a.e. in }Q.
\end{equation}
The initial conditions are
\begin{align*}
\bu^m(0) = P^m\bu_0, \quad \dot{\bu}^m(0) = P^m \bv_0,
\end{align*}
where \( P^m \) is the orthogonal projection operator from \( L^2_\#(\Omega)^d\) to the space of \( \mathbb{R}^d\)-valued trigonometric polynomials of degree at most \( m \). We note that the restriction of \( P^m \) to \( L^2_*(\Omega)^d\) coincides with the orthogonal projection operator from \( L^2_*(\Omega)^d\) to \( V_m \).
We will refer to this finite-dimensional problem as the Galerkin approximation from \( V_m\).

For every \( \bv\in L^2_*(\Omega)^d\) we may identify \( P^m\bv\) with a vector \((v_{ij})_{i,j=1}^{M_m,d}\) such that
\[
P^m\bv = \sum_{i=1}^{M_m}\sum_{j=1}^d v_{ij}\phi_i \be_j.
\]
With this in mind,  the Galerkin approximation from \( V_m \) can be rewritten as the following system of first order ODEs:
\begin{equation}\label{pp:equ7}
\begin{aligned}
(\dot{\betab}^m(t), \dot{\gammab}^m(t) ) & = (\gammab^m(t), \mathbf{g}(t, \betab^m(t), \gammab^m(t))),
\\
\betab^m(0) &= P^m\bu_0,
\\
\gammab^m(0) &= P^m\bv_0,
\end{aligned}
\end{equation}
where the function \( \mathbf{g} = (g_{kl})_{k,l=1}^{M_m,d}\) is defined by
\begin{equation}\nonumber
g_{kl}(t, \betab,\gammab) = -\int_\Omega F_n^{-1}\bigg( \beps\Big( \sum_{i=1}^{M_m}\sum_{j=1}^d (\gamma_{ij} + \alpha\beta_{ij}) \phi_i \be_j \Big) \bigg): \beps(\phi_k\be_l) \,\mathrm{d}x + \int_\Omega \boldf(t)\cdot (\phi_k \be_l) \,\mathrm{d}x.
\end{equation}
The first term is independent of \( t \) and is  continuous in \( (\betab,\gammab) \) by the continuity of \( F_n^{-1}\).
The second term is independent of \(( \betab,\gammab) \) and is measurable with respect to  \( t\) thanks to the assumptions on \( \boldf\). Thus we may apply standard Carath\'{e}odory theory, for example Theorem 2.4.1 from \cite{optimalcontrol}, to deduce that a solution exists to (\ref{pp:equ7}) on \([0, T_*) \) for some positive  \( T_*\leq T \) that may depend on \(m \) and \( n \).
Thus a solution exists to the Galerkin approximation from \( V_m \) on a, possibly  small,  time interval  \( [0, T_*) \). In order to extend the existence result to \([0, T]\) we will deduce an energy inequality and combine this with the fact that
\begin{equation}\label{pp:equ8}
\|\bu^m(t) \|_2^2 + \|\dot{\bu}^m(t) \|_2^2 = \sum_{i=1}^{M_m}\sum_{j=1}^d \Big( |\beta^m_{ij}(t)|^2 + |\dot{\beta}^m_{ij}(t) |^2 \Big),
\end{equation}
which follows from the orthonormality of the basis. We note that  \( \boldf\) can be extended to the interval \( (0 , T+ {\epsilon}) \) and the above reasoning shows that the solution can then be extended onto the interval \( (0, T+ {\epsilon}) \). This fact is need in the proof of Theorem \ref{pp:thm2}.

We multiply (\ref{pp:equ6}) by \( (\dot{\beta}^m_{kl}+\alpha\beta^m_{kl})(t)\) and sum over \( k\in \{ 1,\dots,M_m\} \) and \( l\in \{ 1,\dots,d\} \) to deduce that
\begin{equation}\nonumber
\begin{aligned}
0
&=
\int_\Omega \ddot{\bu}^m \cdot(\dot{\bu}^m + \alpha\bu^m) + \bbT^m:\beps(\dot{\bu}^m + \alpha\bu^m) - \boldf\cdot (\dot{\bu}^m + \alpha\bu^m) \,\mathrm{d}x
\\
&= \int_\Omega \frac{\partial}{\partial t}\Big( \frac{|\dot{\bu}^m|^2}{2} + \alpha\dot{\bu}^m \cdot \bu^m \Big) - \alpha|\dot{\bu}^m |^2 + \bbT^m: F_n(\bbT^m)  - \boldf\cdot (\dot{\bu}^m + \alpha\bu^m) \,\mathrm{d}x.
\end{aligned}
\end{equation}
For arbitrary \( t\in (0, T_*) \) we integrate over \( (0, t) \) to find that
\begin{equation}\label{pp:equ9}
\begin{aligned}
&\frac{\|\dot{\bu}^m (t) \|_2^2 }{2} + \int_0^t\int_\Omega \frac{|\bbT^m|^2 }{(1 + |\bbT^m|^a)^{\frac{1}{a}}} + \frac{|\bbT^m|^2}{n( 1 + |\bbT|^{1-\frac{1}{n}})}\,\mathrm{d}x\,\mathrm{d}s
\\
&\quad\quad
= \frac{\|\dot{\bu}^m(0)\|_2^2}{2} + \alpha\int_\Omega\dot{\bu}^m(0) \cdot \bu^m(0) - \dot{\bu}^m(t) \cdot \bu^m(t) \,\mathrm{d}x + \alpha\int_0^t\int_\Omega |\dot{\bu}^m|^2\,\mathrm{d}x\,\mathrm{d}s
\\
&\quad\quad\quad\quad
+ \int_0^t \int_\Omega \boldf \cdot (\dot{\bu}^m + \alpha\bu^m) \,\mathrm{d}x\,\mathrm{d}s
\\
&\quad\quad
\leq \Big( \frac{\alpha+1}{2}\Big) \|P^m\bv_0 \|_2^2 + \Big( \frac{\alpha}{2}+ 2\alpha^2 \Big) \|P^m\bu_0 \|_2^2 + (2\alpha^2 + \alpha + 1) \int_0^t\int_\Omega |\dot{\bu}^m|^2 \,\mathrm{d}x\,\mathrm{d}s
\\&\quad\quad\quad\quad
+ \frac{\|\dot{\bu}^m(t) \|_2^2}{4} + \alpha^2 \int_0^t\int_\Omega |\bu^m(t) |^2\,\mathrm{d}x\,\mathrm{d}s + \frac{1}{2}\int_0^t\int_\Omega|\boldf|^2\,\mathrm{d}x\,\mathrm{ds}.
\end{aligned}
\end{equation}
However, we have that
\begin{equation}\label{pp:equ10}
\bu^m(t) = P^m\bu_0 + \int_0^t \dot{\bu}^m(s) \,\mathrm{d}s,
\end{equation}
for every \( t\in (0, T_*) \). Thus we get
\begin{align*}
\int_0^t\int_\Omega |\bu^m|^2\,\mathrm{d}x\,\mathrm{d}s
&\leq
2\int_0^t\int_\Omega |P^m\bu_0|^2\,\mathrm{d}x\,\mathrm{d}s
+ 2\int_0^t\int_\Omega \Big|\int_0^s \dot{\bu}^m(\tau, x)\,\mathrm{d}\tau \Big|^2 \,\mathrm{d}x\,\mathrm{d}s
\\
&\leq
2t\|P^m\bu_0 \|_2^2 + 2\int_0^t\int_\Omega\Big( s\int_0^s|\dot{\bu}^m(\tau)|^2\,\mathrm{d}\tau\Big)\,\mathrm{d}x\,\mathrm{d}s
\\
&\leq
2t\|P^m\bu_0 \|_2^2 + 2t^2 \int_0^t \int_\Omega |\dot{\bu}^m(s) |^2 \,\mathrm{d}x\,\mathrm{d}s.
\end{align*}
Substituting this into the right-hand side of (\ref{pp:equ9}), we deduce that
\begin{align*}
&\frac{\|\dot{\bu}^m(t) \|_2^2 }{4} + \int_0^t\int_\Omega \frac{|\bbT^m|^2 }{(1 + |\bbT^m|^a)^{\frac{1}{a}}} + \frac{|\bbT^m|^{2}}{n( 1 + |\bbT|^{1-\frac{1}{n}})}\,\mathrm{d}x\,\mathrm{d}s
\\
&\quad\quad
\leq
\Big( \frac{\alpha+1}{2}\Big) \|P^m\bv_0 \|_2^2 + \Big( \frac{\alpha}{2} + 2\alpha^2 + 2T_*\Big) \|P^m\bu_0\|_2^2
\\&\quad\quad\quad\quad
+ (2\alpha^2 + \alpha + 2T_*^2 + 1) \int_0^t\int_\Omega|\dot{\bu}^m|^2\,\mathrm{d}x\,\mathrm{d}s + \frac{1}{2}\int_0^t\int_\Omega|\boldf|^2\,\mathrm{d}x\,\mathrm{d}s
\\
&\quad\quad
\leq
\Big( \frac{\alpha+1}{2}\Big) \|\bv_0\|_2^2 + \Big( \frac{\alpha}{2} + 2\alpha^2 + 2T + 1\Big) \|\bu_0\|_2^2
\\
&\quad\quad\quad\quad
+ 4(2\alpha^2 + \alpha +2(T+1)^2 + 1)\int_0^t\int_\Omega\frac{|\dot{\bu}^m|^2}{4} \,\mathrm{d}x\,\mathrm{d}s + \frac{1}{2}\int_Q |\boldf|^2 \,\mathrm{d}x\,\mathrm{d}t.
\end{align*}
Applying Gr\"{o}nwall's inequality, together with Lemma \ref{pp:lem1} to the second term on the left-hand side, and taking the supremum over \( [0, T_*) \), we deduce that
\begin{equation}\label{pp:equ11}
\begin{aligned}
&\sup_{t\in[0,T_*)} \|\dot{\bu}^m(t) \|_2^2 + \int_0^{T_*}\int_\Omega|\bbT^m| + \frac{|\bbT^m|^{ 1 + \frac{1}{n}}}{n}\,\mathrm{d}x\,\mathrm{d}t
\\&\quad\quad\leq C(a, \alpha, T) \Big( \|\bu_0 \|_2^2 + \|\bv_0\|_2^2 + \|\boldf\|_{L^2(Q)}^2 \Big),
\end{aligned}
\end{equation}
where \( C\) is a finite constant depending only on \( a\), \( \alpha\) and \( T\).
Using (\ref{pp:equ10}), (\ref{pp:equ11}) and (\ref{pp:equ8}), we deduce that
\begin{align*}
\sup_{t\in [0, T_*)}\max_{i,j}|\beta^m_{ij}(t) |^2 + \sup_{t\in [0, T_*)}\max_{i,j}|\dot{\beta}^m_{ij}(t) |^2 \leq C,
\end{align*}
where \( C\) is the right-hand side of (\ref{pp:equ11}). Thus we may repeatedly apply the Carath\'{e}odory existence theorem to deduce the existence of a solution \( (\bu^m,\bbT^m) \) to the Galerkin approximation from \( V_m\) on the whole of \( [0, T]\).
Repeating the above reasoning and using (\ref{pp:equ10}), we deduce that
\begin{equation}\label{pp:equ12}
\begin{aligned}
&\sup_{t\in[0,T]} \|{\bu}^m(t) \|_2^2+ \sup_{t\in[0,T]} \|\dot{\bu}^m(t) \|_2^2 + \int_Q|\bbT^m| + \frac{|\bbT^m|^{1 + \frac{1}{n}}}{n}\,\mathrm{d}x\,\mathrm{d}t
\\&\quad\quad
\leq C(a, \alpha, T) \Big( \|\bu_0 \|_2^2 + \|\bv_0\|_2^2 + \|\boldf\|_{L^2(Q)}^2 \Big).
\end{aligned}
\end{equation}
From (\ref{pp:equ12}), we can immediately deduce a further estimate. Using (\ref{pp:equ13}), we get
\begin{equation}\label{pp:equ14}
\begin{aligned}
\Big(\int_Q |\beps(\dot{\bu}^m + \alpha\bu^m ) |^{n+1}\,\mathrm{d}x\,\mathrm{d}t\Big)^{\frac{1}{n+1}}
&
\leq |Q|^{\frac{1}{n+1}} + \Big( \int_Q \frac{|\bbT^m|^{1 + \frac{1}{n}}}{n^{n+1}}\,\mathrm{d}x\,\mathrm{d}t\Big)^\frac{1}{n+1}
\leq C,
\end{aligned}
\end{equation}
where \( C\) is a positive constant depending only on \( a\), \( \alpha\), \( T\), \( d\) and the data. Applying Theorem \ref{pp:lem4}, it follows that \( (\dot{\bu}^m + \alpha\bu^m)_m\) is bounded in \( L^{n+1}(0, T; W^{1,n+1}_*(\Omega)^d) \), independent of \( m \).

The bound on \( \bbT^m \) in (\ref{pp:equ12}) allows us to deduce a bound on \( \ddot{\bu}^m \) in \( L^{n+1}(0, T; (W^{1,n+1}_*(\Omega)^d)^\prime) \) by use of (\ref{pp:equ6}).
However, since eventually we will let \( n\rightarrow\infty\) it would be preferable to obtain a bound in a space that does not depend on \( n \).
To this end, we now multiply (\ref{pp:equ6}) by \( (\ddot{\beta}^m_{ij} + \alpha\dot{\beta}^m_{ij})(t)\) and sum over \( k \in \{ 1,\dots ,d \} \) and \( l \in \{1,\dots, d\} \) to deduce that
\begin{equation}\label{pp:equ15}
0 = \int_\Omega |\ddot{\bu}^m|^2 + \frac{\partial}{\partial t}\left( \frac{\alpha}{2}|\dot{\bu}^m|^2 + \frac{h_n(|\bbT^m|^2)}{2}\right)  - \boldf\cdot (\ddot{\bu}^m + \alpha\dot{\bu}^m ) \,\mathrm{d}x,
\end{equation}
where \( h_n : [0, \infty) \rightarrow[0, \infty) \) is defined by
\[
h_n(s) = \int_0^s \frac{1}{(1 + t^{\frac{a}{2}})^{1 + \frac{1}{a}}} +
\frac{1}{n^2( 1 + t^{\frac{1}{2}-\frac{1}{2n}} )^2}
+ \left( 1 - \frac{1}{n}\right) \frac{1}{n( 1 + t^{\frac{1}{2}- \frac{1}{2n}})^2}
\,\mathrm{d}t,
\]
where we have made use of the following reasoning:
\begin{align*}
\bbT^m : \frac{\partial}{\partial t} \Big( F_n(\bbT^m)\Big)
&= \bbT^m : \bigg[ \frac{\dot{\bbT}^m}{( 1 + |\bbT^m|^a)^{\frac{1}{a}}} + \frac{\dot{\bbT}^m}{n(1 + |\bbT^m|^{1-\frac{1}{n}}) } -
\frac{(\bbT^m : \dot{\bbT}^m) |\bbT^m|^{a-2}\bbT^m}{(  1 + |\bbT^m|^a)^{1 + \frac{1}{a}}}
\\&\quad\quad\quad\quad - \bigg( 1 - \frac{1}{n}\bigg) \frac{(\bbT^m:\dot{\bbT}^m) |\bbT^m|^{-1-\frac{1}{n} } \bbT^m}{n(1 + |\bbT^m|^{1-\frac{1}{n}})^2}\bigg]
\\
&= \frac{\bbT^m : \dot{\bbT}^m}{( 1 + |\bbT^m|^a)^{1 + \frac{1}{a}}} + \frac{\bbT^m : \dot{\bbT}^m}{n^2 ( 1 + |\bbT^m|^{1-\frac{1}{n}})} + \bigg( 1 - \frac{1}{n}\bigg) \frac{\bbT^m : \dot{\bbT}^m }{n( 1 + |\bbT^m|^{ 1 - \frac{1}{n}} )^2}
\\
&= h_n^\prime(|\bbT^m|^2) \bbT^m : \dot{\bbT}^m
\\
&= h_n^\prime(|\bbT^m|^2)\frac{\partial}{\partial t}\left(\frac{|\bbT^m|^2}{2}\right)
\\
&=\frac{\partial}{\partial t}\left(\frac{h_n(|\bbT^m|^2)}{2}\right).
\end{align*}
Using Lemma \ref{pp:lem1}, there exist positive constants \(c_a\), \( C_a\) depending only on \( a\) such that
\begin{equation}\label{pp:equ16}
c_a\Big( s^{\frac{1}{2}- \frac{a}{2}}\chi_{\{ s \geq 1\} } -1 \Big)
\leq h_n(s) \leq
s^{\frac{1}{n}}+ C_a\Big( s^{\frac{1}{2}- \frac{a}{2}}\chi_{\{ s \geq 1\} } +1 \Big).
\end{equation}
We require the indicator functions since we do not necessarily have \( a\in (0, 1]\).
Integrating (\ref{pp:equ15}) over \( (0, t) \) for an arbitrary \( t\in (0, T) \) and using (\ref{pp:equ16}), we deduce that
\begin{align*}
&\int_0^t \int_\Omega|\ddot{\bu}^m|^2 \,\mathrm{d}x\,\mathrm{d}s + \frac{\alpha}{2}\|\dot{\bu}^m(t) \|_2^2+ \int_\Omega c_a|\bbT^m(t)|^{1-a}\chi_{\{|\bbT^m(t)|\geq 1\} } \,\mathrm{d}x
\\
&\quad\quad
\leq \frac{\alpha}{2}\|\dot{\bu}^m(0) \|_2^2 + \int_\Omega C_a |\bbT^m(0)|^{1-a}\chi_{\{|\bbT^m(0)|\geq 1\}} + |\bbT^m(0)|^{\frac{2}{n}}\,\mathrm{d}x
\\&\quad\quad\quad\quad
+ \int_0^t\int_\Omega \boldf\cdot (\ddot{\bu}^m + \alpha\dot{\bu}^m) \,\mathrm{d}x\,\mathrm{d}t
+ C(a,\Omega)
\\
&\quad\quad
\leq \frac{\alpha}{2}\|\bv_0 \|_2^2 + (C_a + 1) \int_\Omega|\bbT^m(0)|^2\,\mathrm{d}x + \frac{1}{2}\int_0^t\int_\Omega |\ddot{\bu}^m|^2 \,\mathrm{d}x\,\mathrm{d}s
+ \int_0^t\int_\Omega |\boldf|^2\,\mathrm{d}x\,\mathrm{d}s
\\&\quad\quad\quad\quad
+ \frac{\alpha^2}{2}\int_0^t \int_\Omega |\dot{\bu}^m|^2 \,\mathrm{d}x\,\mathrm{d}s + C(a,\alpha, \Omega).
\end{align*}
Rearranging, applying (\ref{pp:equ12}) and taking the supremum over \( t\in (0, T) \), it follows that
\begin{equation}\label{pp:equ18}
\begin{aligned}
&\frac{1}{2}\int_Q |\ddot{\bu}^m|^2\,\mathrm{d}x\,\mathrm{d}t
+\sup_{t\in [0, T]}\Big(
\int_\Omega C_a |\bbT^m(t)|^{1-a}\chi_{\{|\bbT^m(t)|\geq 1\}} \,\mathrm{d}x\Big)
\\
&\quad\quad
\leq C(a,\alpha, \Omega, T) \Big( \|\bu_0 \|_2^2 + \|\bv_0 \|_2^2 + \|\boldf\|_{L^2(Q)}^2 + \int_\Omega |\bbT^m(0)|^2\,\mathrm{d}x\Big),
\end{aligned}
\end{equation}
where \( C\) is a positive constant that is independent of \( m \) and \( n \). To show that the sequence \( (\bbT^m(0))_m \) is bounded in \( L^2(\Omega)^{d\times d}\), independent of \( m \) and \( n\), we first note that
\begin{align*}
\beps(P^m(\bv_0 + \alpha\bu_0 )) = \frac{\bbT^m(0)}{(1 + |\bbT^m(0)|^a)^{\frac{1}{a}}}  + \frac{\bbT^m(0)}{n( 1 + |\bbT^m(0)|^{1-\frac{1}{n}})}.
\end{align*}
Suppose that there exist \(m_0 \in \mathbb{N}\) and \( C_1\in (0, 1) \) such that \( \|\beps(P^m(\bv_0 + \alpha\bu_0))\|_\infty\leq C_1\) for every \( m\geq m_0 \). Since \( |F^{-1}(\bbT)| \geq |F_n^{-1}(\bbT)|\) for every \( \bbT \in \mathbb{R}^{d\times d}\) with \( |\bbT|< 1\) and \( F^{-1}\) is a radial function that increases in absolute value as \( |\bbT |\) increases, we deduce that
\[
|F_n^{-1}(\beps(P^m(\bv_0 + \alpha\bu_0)))|\leq |F^{-1}(\beps(P^m(\bv_0 + \alpha\bu_0))) | \leq f^{-1}(C_1) < \infty,
\]
a.e. in \( \Omega\),  where \( f\) is a function from \( [0, \infty ) \) to \( [0, 1) \) defined by
\[
f(s) = \frac{s}{(1 + s^a)^{\frac{1}{a}}}.
\]
It follows that \( (\bbT^m(0))_{m\geq m_0} \) is uniformly bounded in \( L^\infty(\Omega)^{d\times d}\) independent of \( m \) and \( n \).

It remains to prove the existence of such an \( m_0 \) and \( C_1\). From standard properties of projection operators, for every \( \bv\in L^2_\#(\Omega)^d\) we have \( P^m \bv \rightarrow\bv \) strongly in \( L^2_\#(\Omega)^d\) as \( m \rightarrow\infty \). Furthermore, the projection operator commutes with derivation \cite{RN115}, i.e., for every \( \bv\in  W^{1,2}_\#(\Omega)^d\) we have \( \nabla (P^m\bv) = P^m(\nabla \bv) \) (assuming that \(P^m\) acts component-wise on matrix-valued functions). Thus, for every \( \bv\in W^{k,2}_*(\Omega)^d\) and \( k \in \mathbb{N}\), the following holds:
\begin{align*}
\lim_{m\rightarrow\infty}\|P^m\bv - \bv\|_{k,2} = 0.
\end{align*}
Applying the Sobolev embedding theorem, if \( \bv\in W^{k+1,2}_*(\Omega)^d\) for  \( k \geq\frac{d}{2}\), we have
\begin{align*}
\|\beps(P^m\bv) - \beps(\bv) \|_\infty &\leq C\|\beps(P^m\bv) - \beps(\bv)\|_{k,2}
\\
&\leq C\|P^m \bv - \bv\|_{k+1,2},
\end{align*}
where the right-hand side vanishes in the limit as \( m\rightarrow\infty \) and \( C\) is independent of \( \bv\). Setting \( \bv = \bv_0 + \alpha\bu_0 \), we deduce that there exists an \( m_0 \in \mathbb{N}\) such that
\[
\|\beps(P^m(\bv_0 + \alpha\bu_0)) - \beps(\bv_0 + \alpha\bu_0 ) \|_\infty \leq \frac{1 - C_*}{2},
\]
for every \( m\geq m_0\). The right-hand side is positive since \( C_* \in (0, 1) \). Using (\ref{pp:equ17}) it follows that
\begin{equation}\nonumber
\|\beps(P^m(\bv_0 + \alpha\bu_0))\|_\infty \leq \frac{1 + C_*}{2}=: C_1 < 1,
\end{equation}
for every \( m\geq m_0 \). Substituting this into (\ref{pp:equ18}) we get
\begin{equation}\label{pp:equ19}
\begin{aligned}
&\int_Q |\ddot{\bu}^m|^2 \,\mathrm{d}x\,\mathrm{d}t + \sup_{t\in[0, T]}\Big( \int_\Omega |\bbT^m(t) |^{1-a}\chi_{\{|\bbT^m(t)|\geq 1\}}\,\mathrm{d}x\Big)
\\
&\quad\quad
\leq
C(a,\alpha, \Omega, T) \Big( \|\bu_0 \|_2^2 + \|\bv_0 \|_2^2 + \|\boldf\|_{L^2(Q)}^2 + f^{-1}(C_1) \Big)
\\
&\quad\quad
\leq
C(a,\alpha, \Omega, T,\bu_0, \bv_0 ,\boldf),
\end{aligned}
\end{equation}
for every \( m\geq m_0 \), where \( C\) is a positive constant that is independent of \( m \) and \( n \).

Putting together (\ref{pp:equ12}), (\ref{pp:equ14}), (\ref{pp:equ19}) and using Korn's inequality, we deduce the following convergence results up to a subsequence that we do not relabel:
\begin{itemize}
\item \( \bu^m \overset{\ast}{\rightharpoonup} \bu\) weakly-* in \( L^\infty(0, T; L^2(\Omega)^d)\);
\item \( \dot{\bu}^m \overset{\ast}{\rightharpoonup} \dot{\bu}\) weakly-* in \( L^\infty(0, T; L^2(\Omega)^d)\);
\item \( \ddot{\bu}^m\rightharpoonup \ddot{\bu}\) weakly in \( L^2(0, T; L^2_*(\Omega)^d) \);
\item \( \dot{\bu}^m + \alpha\bu^m \rightharpoonup \dot{\bu}+\alpha\bu\) weakly in \( L^{n+1}(0, T; W^{1,n+1}_*(\Omega)^d) \);
\item \( \bbT^m \rightharpoonup \bbT\) weakly in \( L^{1 + \frac{1}{n}}(0, T; L^{1 + \frac{1}{n}}_\#(\Omega)^{d\times d}) \).
\end{itemize}
We claim that \( (\bu, \bbT) \) is a weak solution of (\ref{pp:equ23}).

By standard regularity results, we know that \( \bu\), \( \dot{\bu}\in C([0, T]; L^2(\Omega)^d) \) up to redefinition almost everywhere. In particular, we have
\begin{equation}\label{pp:equ20}
\lim_{t\rightarrow 0+ }\Big( \|\bu(t) - \bu(0) \|_2  + \|\dot{\bu}(t) - \dot{\bu}(0) \|_2\Big) = 0.
\end{equation}
However, by the Aubin--Lions lemma, we know that the sequences \( (\bu^m)_m\), \((\dot{\bu}^m)_m \) converge strongly in \( C([0, T]; (W^{1,2}_*(\Omega)^d)^\prime)\). Thus we have
\begin{align*}
0 &= \lim_{m\rightarrow\infty} \|\bu(0)  - \bu^m(0) \|_{-1,2} + \|\dot{\bu}(0) - \dot{\bu}^m(0) \|_{-1,2}
\\
&= \lim_{m\rightarrow\infty}\|\bu(0) - P^m\bu_0 \|_{-1,2} + \|\dot{\bu}(0) - P^m \bv_0 \|_{-1,2}
\\
&= \|\bu(0) - \bu_0 \|_{-1,2} + \|\dot{\bu}(0) - \bv_0 \|_{-1,2},
\end{align*}
where \( \|\cdot\|_{-1,2}\) denotes the norm in \( (W^{1,2}_*(\Omega)^d)^\prime\). By combining this with (\ref{pp:equ20}),  the assertion (\ref{pp:equ23}) immediately follows.

To show that (\ref{pp:equ21}) holds, first note that for every \( \bv\in C^\infty_*(\overline{\Omega})^d\) and \( t\in (0, T) \) we have
\begin{equation}\label{pp:equ24}
\int_\Omega\ddot{\bu}^m(t) \cdot P^m\bv + \bbT^m(t) : \beps(P^m\bv) \,\mathrm{d}x = \int_\Omega \boldf(t) \cdot P^m\bv \,\mathrm{d}x.
\end{equation}
For an arbitrary but fixed \( \psi \in C([0, T]) \),  we multiply (\ref{pp:equ24}) by \( \psi(t) \) and integrate over \( (0, T) \). Noting that \( P^m \bv\) converges strongly in \( W^{k,2}_*(\Omega)^d\) to \( \bv\) for every \( k \in \mathbb{N}\), letting \( m\rightarrow\infty \) we get
\begin{align*}
0 &=\int_Q \ddot{\bu}\cdot (\psi \bv) + \bbT: \beps(\psi \bv) - \boldf\cdot (\psi\bv)\,\mathrm{d}x\,\mathrm{dt}
\\
&= \int_0^T \psi(t) \cdot\Big( \int_\Omega\ddot{\bu}(t) \cdot \bv + \bbT(t) :\beps(\bv) - \boldf(t) \cdot \bv\,\mathrm{d}x\Big) \,\mathrm{d}t.
\end{align*}
Since \( \psi \) is arbitrary,  the second factor is integrable over \( (0, T) \) and \( C^\infty_*(\overline{\Omega})^d\) is dense in \( W^{1,p}_*(\Omega)^d\) for every \( p \in [1,\infty) \),  we deduce that (\ref{pp:equ21}) holds. Furthermore, we note that \( (\dot{\bu}+\alpha\bu)(t)\) is a valid test function in (\ref{pp:equ21}) for a.e. \( t\in (0, T) \).

To show that (\ref{pp:equ22}) holds, we will use a variant of Minty's method.
First note that the sequence \( (\dot{\bu}^m + \alpha\bu^m )_m\) converges weakly in \( L^2(0, T; W^{1,2}_*(\Omega)^d) \) and \( (\ddot{\bu}^m + \alpha\dot{\bu}^m)_m \) converges weakly in \( L^2(0, T; L^2_*(\Omega)^d) \). Thus by the Aubin--Lions lemma, the sequence \( (\dot{\bu}^m + \alpha\bu^m)_m \) converges strongly in the space \( L^2(0, T; L^2_*(\Omega)^d) \) as \( m\rightarrow\infty \). Testing (\ref{pp:equ6}) against \((\dot{\bu}^m + \alpha\bu^m)(t)\) and integrating over \( (0, T) \) we deduce that
\begin{equation}\label{pp:equ25}
\begin{aligned}
\lim_{m\rightarrow\infty}\int_Q \bbT^m:\beps(\dot{\bu}^m + \alpha\bu^m ) \,\mathrm{d}x\,\mathrm{d}t
&= \lim_{m\rightarrow\infty}\int_Q - \ddot{\bu}^m\cdot (\dot{\bu}^m + \alpha\bu^m ) + \boldf\cdot (\dot{\bu}^m + \alpha\bu^m ) \,\mathrm{d}x\,\mathrm{dt}
\\
&= \int_Q - \ddot{\bu}\cdot (\dot{\bu}+\alpha\bu) + \boldf\cdot (\dot{\bu}+\alpha\bu) \,\mathrm{d}x\,\mathrm{d}t
\\
&= \int_Q \bbT: \beps(\dot{\bu}+\alpha\bu) \,\mathrm{d}x\,\mathrm{d}t,
\end{aligned}
\end{equation}
using that \( (\dot{\bu}+\alpha\bu)(t) \) is a valid test function in (\ref{pp:equ21}) for the transition to the final line.
Let \( \bbS \in L^{1 + \frac{1}{n}}(Q)^{d\times d}\) be arbitrary but fixed. Using the monotonicity of \( F_n \), (\ref{pp:equ25}) and the convergence results, we have that
\begin{equation}\label{pp:equ26}
\begin{aligned}
0 &\leq \lim_{m\rightarrow\infty} \int_Q (\bbT^m - \bbS):(F_n(\bbT^m) - F_n(\bbS) ) \,\mathrm{d}x\,\mathrm{d}t
\\
&= \int_Q (\bbT - \bbS) :(\beps(\dot{\bu}+\alpha\bu ) - F_n(\bbS))\,\mathrm{d}x\,\mathrm{d}t.
\end{aligned}
\end{equation}
We replace \(\bbS\) by \( \bbT \pm \gamma\bbU\) for an arbitrary \( \gamma> 0 \) and \( \bbU\in L^\infty(Q)^{d\times d}\)  to obtain
\begin{align*}
0 \leq \mp \int_Q \gamma\bbU:(\beps(\dot{\bu}+\alpha\bu) - F_n(\bbT \pm \gamma\bbU) ) \,\mathrm{d}x\,\mathrm{d}t.
\end{align*}
We divide through by \( \gamma\) and use Lebesgue's dominated convergence theorem when letting \( \gamma\rightarrow 0+ \) in order to deduce that
\begin{align*}
 0 \leq \mp \int_Q \bbU: (\beps(\dot{\bu}+\alpha\bu)  - F_n(\bbT)) \,\mathrm{d}x\,\mathrm{d}t.
\end{align*}
Setting \( \bbU = \frac{\beps(\dot{\bu}+\alpha\bu)  - F_n(\bbT)}{1 + |\beps(\dot{\bu}+\alpha\bu)  - F_n(\bbT)|}\), it immediately follows that (\ref{pp:equ22}) holds.
Hence \( (\bu,\bbT) \) is a weak solution of (\ref{pp:equ2}).

To show that (\ref{pp:equ27}) holds, we only need to show that \( \bbT^m \rightarrow\bbT \) converges pointwise a.e. on \( Q\) as \( m\rightarrow\infty \) and \( \bbT^m(t) \rightarrow\bbT(t) \) pointwise a.e. on \( \Omega\) for a.e. \( t\in (0, T) \).  Combining this with (\ref{pp:equ12}), (\ref{pp:equ14}), (\ref{pp:equ19}) and Fatou's lemma with the weak lower semi-continuity of norms, the bound  (\ref{pp:equ27}) will follow.

We can in fact  prove the stronger result that \( \bbT^m \rightarrow\bbT \) strongly in \( L^1(Q)^{d\times d}\) as \( m\rightarrow\infty\) by mimicking an argument contained in \cite{RN8}. For each \( k > 0 \), we define the set
\begin{align*}
Q^m_k = \{ (t,x)\in Q: 1 + |\bbT| + |\bbT^m|> k \}.
\end{align*}
Using that \( \bbT^m \rightharpoonup \bbT \) weakly in \( L^{1 + \frac{1}{n}}(Q)^{d\times d}\) and the bound from (\ref{pp:equ12}), there exists a positive constant \( C = C(n) \) independent of \( m \) such that
\begin{align*}
\int_Q |\bbT|^{1 + \frac{1}{n}} + |\bbT^m|^{1 + \frac{1}{n}} \,\mathrm{d}x\,\mathrm{d}t \leq C(n).
\end{align*}
It follows that \( |Q^m_k |\leq C(n) k^{-( 1 + \frac{1}{n})}\). With this in mind, we have that
\begin{align*}
&\Big( \int_Q |\bbT^m - \bbT|\,\mathrm{d}x\,\mathrm{d}t\Big)
\\
&\quad\quad
\leq C\|\bbT^m - \bbT\|_{L^{1 + \frac{1}{n}}(Q^m_k)}|Q^m_k|^{\frac{2}{n+1}}
+Ck^{1 + a}\int_{Q\setminus Q^m_k} \frac{|\bbT^m - \bbT|^2}{(1 + |\bbT^m| + |\bbT|)^{1 + a}}\,\mathrm{d}x\,\mathrm{d}t
\\
&\quad\quad
\leq Ck^{-\frac{2}{n}} + Ck^{1 + a} \int_Q (\bbT^m - \bbT):(F(\bbT^m) - F(\bbT))\,\mathrm{d}x\,\mathrm{d}t
\\
&\quad\quad
\leq
Ck^{-\frac{2}{n}} + Ck^{1 + a}\int_Q (\bbT^m - \bbT) : (\beps(\dot{\bu}^m + \alpha\bu^m) - \beps(\dot{\bu}+\alpha\bu)) \,\mathrm{d}x\,\mathrm{d}t,
\end{align*}
where \( C\) is a positive constant that is independent of \( k \) and \( m \).
In the limit as \( m\rightarrow\infty\), the second term on the right-hand side will vanish, recalling (\ref{pp:equ25}). It follows that, for every \( k > 0 \),
\begin{align*}
\lim_{m\rightarrow\infty }\Big( \int_Q |\bbT^m - \bbT |\,\mathrm{d}x\,\mathrm{d}t\Big)^2 \leq Ck^{-\frac{2}{n}}.
\end{align*}
Since \( k \) is arbitrary, we deduce that \( \bbT^m \rightarrow\bbT \) strongly in \( L^1(Q)^{d\times d}\) as \( m\rightarrow\infty \). Taking a further subsequence if necessary, we get that \( \bbT^m \rightarrow\bbT \) pointwise a.e. in \( Q\).

To prove that \( \bbT^m(t) \rightarrow\bbT(t) \) converges pointwise a.e. on \( \Omega\) for a.e. \( t\in (0, T) \), suppose otherwise. That is, assume that there exists a measurable set \( A\subset (0, T) \) of positive measure such that, for each \( t\in A\), there exists a measurable set \( B(t) \subset \Omega\) of positive measure such that \( (\bbT^m(t,x))_m\) does not converge to \( \bbT(t,x) \), for every \( x\in B(t) \). Let \( M = \{ (t,x) : t\in A, x\in B(t) \}\), a measurable subset of \( Q\) such that
\[
|M| = \int_A\int_{B(t)} 1\,\mathrm{d}x\,\mathrm{d}t = \int_A |B(t) |\,\mathrm{d}t > 0.
\]
However, \( \bbT^m \not\to \bbT \) pointwise on \( M \). This contradicts the fact that \( \bbT^m\rightarrow\bbT \) pointwise a.e. on \( Q\). Thus our original claim holds and we deduce that (\ref{pp:equ27}) holds. This concludes the existence part of the proof.

To prove uniqueness, suppose that \( (\bu_1, \bbT_1) \), \( (\bu_2,\bbT_2) \) are weak solutions of (\ref{pp:equ2}) in the sense of Definition \ref{pp:def2}  with respect to the same initial data.
Let \( \bv := \bu_1 - \bu_2\) and \( \bbS:= \bbT_1 - \bbT_2\).
Testing in (\ref{pp:equ21}) for \( (\bu_i,\bbT_i)\), \( i\in \{ 1,2\}\),  with  test function \( (\dot{\bv} + \alpha\bv)(t)\), integrating over \( (0, t) \) for an arbitrary \( t\in (0, T) \) and subtracting the result, we get
\begin{align*}
0 &= \int_0^t \int_\Omega \ddot{\bv}\cdot (\dot{\bv} + \alpha\bv) + \bbS: \beps(\dot{\bv} + \alpha\bv) \,\mathrm{d}x\,\mathrm{d}s
\\
&= \int_\Omega \frac{|\dot{\bv}(t) |^2}{2} + \alpha\dot{\bv}(t) \cdot \bv(t) \,\mathrm{d}x + \int_0^t \int_\Omega\bbS : \beps(\dot{\bv} + \alpha\bv) - \alpha|\dot{\bv}|^2 \,\mathrm{d}x\,\mathrm{d}s.
\end{align*}
We used the fact that \( \bv(t) \rightarrow \mathbf{0}\) and \( \dot{\bv}(t) \rightarrow\mathbf{0}\) strongly in \( L^2(\Omega)^d\) as \( t\rightarrow 0+\), as well as \( \bv\), \( \dot{\bv}\in C([0, T]; L^2(\Omega)^d) \) with Lemma 7.3 of \cite{nonlinpderoubicek}. Using the monotonicity of \( F_n \), it follows that
\begin{align*}
0 &\leq \int_\Omega \frac{|\dot{\bv}(t) |^2}{2}\,\mathrm{d}x + \int_0^t \int_\Omega (\bbT_1 - \bbT_2):(F_n(\bbT_1) - F_n(\bbT_2))\,\mathrm{d}x\,\mathrm{d}s
\\
&\leq \int_0^t\int_\Omega (2\alpha^2 + \alpha) |\dot{\bv}|^2 \,\mathrm{d}x\,\mathrm{d}s + \int_\Omega\frac{|\dot{\bv}(t) |^2}{4}\,\mathrm{d}x.
\end{align*}
Applying Gr\"{o}nwall's inequality, we deduce that
\[
\int_\Omega\frac{|\dot{\bv}(t) |^2}{4}\,\mathrm{d}x + \int_0^t \int_\Omega (\bbT_1 - \bbT_2):(F_n(\bbT_1) - F_n(\bbT_2))\,\mathrm{d}x\,\mathrm{d}s = 0,
\]
for a.e. \( t\in (0, T) \). Thus \( \dot{\bu}_1 = \dot{\bu}_2 \) a.e. in \( Q\). Since \( \bu_1(0) = \bu_2(0) \), it follows  that \( \bu_1 = \bu_2 \) and \( F_n(\bbT_1) = F_n(\bbT_2) \) a.e. in \( Q\). Noting that \( F_n \) is a bijection, we must have \( \bbT_1 = \bbT_2\) a.e. in \( Q\). In particular, \( (\bu_1,\bbT_1 ) = (\bu_2 ,\bbT_2) \) and we have uniqueness of weak solutions of (\ref{pp:equ2}).
\end{proof}

We would like to use this approximation in order to show that a weak solution exists to the strain-limiting problem (\ref{pp:equ1}). To do this, we must obtain further \textit{a priori} estimates. Since we are working in the periodic setting, we will be able to do this when working with the Galerkin approximation \( (\bu^m ,\bbT^m)\) of the regularised problem, rather than the weak solution of (\ref{pp:equ2}) itself.

\begin{lemma}\label{pp:lem5}
Suppose that the hypotheses of Theorem \ref{pp:thm1} hold and that additionally we have  \( \bu_0 \), \( \bv_0 \in W^{1,2}_*(\Omega)^d\) and \( \boldf \in L^2(0, T; W^{1,2}_*(\Omega)^d) \).
Let \( (\bu^m ,\bbT^m) \) be the solution of the Galerkin approximation from \( V_m \) as in the proof of Theorem \ref{pp:thm1}.
There exists a constant \( C\) independent of \( m \) and \( n \) such that
\begin{equation}\label{pp:equ28}
\begin{aligned}
&\sup_{t\in [0, T]} \|\nabla \bu^m(t) \|_2^2 + \sup_{t\in [0, T]}\|\nabla\dot{\bu}^m(t) \|_2^2 + \int_Q \frac{|\nabla\bbT^m|^2}{(1 + |\bbT^m|)^{1+a}} \,\mathrm{d}x\,\mathrm{d}t
\\&\quad\quad
\leq C(a,\alpha, \Omega, T) \Big( \|\nabla \bu_0 \|_2^2 + \|\nabla \bv_0\|_2^2 + \|\nabla \boldf\|_{L^2(Q)}^2\Big).
\end{aligned}
\end{equation}
Moreover, if we also have that \( \boldf\in W^{1,2}([0, T+\tilde{\epsilon}); L^2_*(\Omega)^d) \) for some \( \tilde{\epsilon}> 0 \), then
\begin{equation}\label{pp:equ29}
\begin{aligned}
&\sup_{t\in [0, T]}\|\ddot{\bu}^m(t) \|_2^2 + \int_Q \frac{|\dot{\bbT}^m|^2}{(1 + |\bbT^m|)^{1 + a}} \,\mathrm{d}x
\\
&\quad\quad
\leq
C(a, \alpha, \Omega, T) \Big( \|\bv_0 \|_2^2 + \|\dot{\boldf}\|_{L^2(Q)}^d
+ \|\boldf(0)\|_{L^2(\Omega)}^2 + (1 + f^{-1}(C_1))\|\bv_0 + \alpha\bu_0 \|_{2,2}^2 \Big),
\end{aligned}
\end{equation}
for every \( m\geq m_0 \), where \( C_1 \in(0, 1) \) and \( m_0 \in\mathbb{N}\) are the constants from the proof of Theorem~\ref{pp:thm1} and \( f\) is defined on \( [0, \infty ) \) by \( f(t) = ( 1 + t^a)^{-\frac{1}{a}}t \).  The constant \( C\) on the right-hand side of (\ref{pp:equ29}) is independent of \( m \) and \( n \).
\end{lemma}

\begin{proof}
We start with (\ref{pp:equ28}). Derivation does not increase the degree of a trigonometric polynomial. In particular, \( \nabla \cdot \nabla (\dot{\bu}^m + \alpha\bu^m ) \) is a trigonometric polynomial of degree at most \( m \). Furthermore it has integral over \( \Omega\) equal to \( 0 \) as a result of  periodicity. Thus \( \nabla \cdot \nabla (\dot{\bu}^m + \alpha\bu^m ) \) is a valid test function in (\ref{pp:equ6}). It follows that
\begin{equation}\label{pp:equ30}
\begin{aligned}
0&=
-\int_\Omega\ddot{\bu}^m \cdot (\nabla\cdot\nabla(\dot{\bu}^m + \alpha\bu^m)) + \bbT^m:\beps(\nabla\cdot\nabla(\dot{\bu}^m + \alpha\bu^m))
\\&\quad\quad
- \boldf\cdot (\nabla\cdot\nabla(\dot{\bu}^m + \alpha\bu^m)) \,\mathrm{d}x
\\
&= \int_\Omega \nabla \ddot{\bu}^m :\nabla (\dot{\bu}^m + \alpha\bu^m) + (\nabla \cdot \bbT^m)\cdot (\nabla\cdot\nabla(\dot{\bu}^m + \alpha\bu^m))
\\&\quad\quad
- \nabla \boldf:\nabla (\dot{\bu}^m + \alpha\bu^m ) \,\mathrm{d}x
\\
&= \int_\Omega \frac{\partial}{\partial t}\Big( \frac{|\nabla \dot{\bu}^m|^2}{2}  + \alpha\nabla \dot{\bu}^m : \nabla \bu^m \Big) - \alpha|\nabla \dot{\bu}^m|^2 + \nabla \bbT^m \,\vdots\, \nabla F_n(\bbT^m)
\\&\quad\quad
- \nabla \boldf:\nabla (\dot{\bu}^m + \alpha\bu^m ) \,\mathrm{d}x.
\end{aligned}
\end{equation}
We use \( \nabla \bbS \) to denote the third order tensor \( (\partial_k S_{ij})_{i,j,k}\) and if \( \mathcal{S}_1\), \(\mathcal{S}_2\) are third order tensors, we let~``\( \vdots\)'' denote the triple scalar product
\[
\mathcal{S}_1 \,\vdots\, \mathcal{S}_2 = \sum_{i,j,k=1}^d (\mathcal{S}_1)_{ijk}(\mathcal{S}_2)_{ijk}.
\]
We justify the transition to the last line of (\ref{pp:equ30}) as follows. For ease of notation, we write \( \bbS = \bbT^m \) and \( \bv = \dot{\bu}^m + \alpha\bu^m \) in the following calculation. Using integration by parts and the periodic boundary conditions, we have
\begin{align*}
\int_\Omega (\nabla\cdot \bbS) \cdot (\nabla \cdot \nabla \bv) \,\mathrm{d}x
&=
\int_\Omega\frac{\partial S_{ij}}{\partial x_j} \frac{\partial^2 v_i}{\partial x_k^2}
\\
&= \int_\Omega\frac{\partial S_{ij}}{\partial x_k}\frac{\partial^2 v_i}{\partial x_k \partial x_j}\,\mathrm{d}x
\\
&= \int_\Omega \frac{\partial S_{ij}}{\partial x_k}\frac{\partial}{\partial x_k}\Big( \frac{1}{2}\Big( \frac{\partial v_i }{\partial x_j} + \frac{\partial v_j }{\partial x_i}\Big) \Big) \,\mathrm{d}x
\\
&= \int_\Omega \nabla \bbS\, \vdots\, \nabla \beps(\bv) \,\mathrm{d}x.
\end{align*}
Manipulating (\ref{pp:equ30}) in an almost identical way as in the calculation that resulted  (\ref{pp:equ12}), we deduce that
\begin{align*}
&\sup_{t\in [0, T]}\|\nabla\bu^m(t) \|_2^2 + \sup_{t\in [0, T]} \|\nabla \dot{\bu}^m(t) \|_2^2 + \int_Q \frac{|\nabla \bbT^m|^2}{( 1 + |\bbT^m|)^{1 + a}} \,\mathrm{d}x\,\mathrm{d}t
\\
&\quad\quad
\leq C(a,\alpha, T) \Big( \|\nabla \bu_0 \|_2^2 + \|\nabla \bv_0 \|_2^2 + \|\nabla \boldf\|_{L^2(Q)}^2 \Big) ,
\end{align*}
where \( C\) is a constant that is independent of \( n \) and \( m \). Thus (\ref{pp:equ28}) holds.

For (\ref{pp:equ29}), we recall that the weak solution \( (\bu^m ,\bbT^m ) \) can be extended to the larger interval \( (0, T+ \epsilon) \) for some \( \epsilon = \epsilon_{m,n}> 0 \). We define the undivided difference quotient with increment \( h \) in the time variable by
\[
\Delta^h_t g(t,x) = g(t+h, x) - g(t,x),
\]
for any function \( g\). Provided that \( h > 0 \) is sufficiently small, we deduce that \( \Delta^h_t \bu^m\), \( \Delta^h_t \bbT^m \)  and \( \Delta^h_t \boldf\) are well-defined on \( (0, T]\). From now on, we will assume that \( h >0 \) is small enough so that this is true.

From (\ref{pp:equ6}) we have that
\begin{equation}\label{pp:equ31}
\int_\Omega \Delta^h_t \ddot{\bu}^m(t) \cdot \bv + \Delta^h_t \bbT^m(t) : \beps(\bv) \,\mathrm{d}x = \int_\Omega \Delta^h_t \boldf(t) \cdot \bv\,\mathrm{d}x,
\end{equation}
for every \( \bv\in V_m\) and \( t\in (0, T) \). Setting \(\bv = \Delta^h_t(\dot{\bu}^m + \alpha\bu^m)(t) \) in (\ref{pp:equ31}), we deduce that
\begin{align*}
0&= \int_\Omega\bigg[ \frac{\partial}{\partial t}\Big( \frac{|\Delta^h_t \dot{\bu}^m|^2}{2} + \alpha\Delta^h_t \dot{\bu}^m \cdot \Delta^h_t \bu^m \Big) - \alpha|\Delta^h_t\dot{\bu}^m|^2 + \Delta^h_t \bbT^m : \Delta^h_t F_n(\bbT^m)
\\
&\quad\quad
 - \Delta^h_t \boldf \cdot \Delta^h_t (\dot{\bu}^m + \alpha\bu^m) \bigg]\,\mathrm{d}x.
\end{align*}
Integrating over \( (0, t) \) for arbitrary \( t\in (0, T) \), manipulating in the usual way and dividing through by \( h^2 \) yields
\begin{align*}
&\frac{\|\Delta^h_t \dot{\bu}^m(t) \|_2^2 }{h^2} + \int_0^t \int_\Omega \frac{\Delta^h_t \bbT^m}{h} :\frac{\Delta^h_t F(\bbT^m) }{h} \,\mathrm{d}x\,\mathrm{d}s
\\
&\quad\quad
\leq C(a, \alpha, T) \Big( \frac{\|\Delta^h_t \dot{\bu}^m(0) \|_2^2}{h^2} + \frac{\|\Delta^h_t \bu^m(0) \|_2^2 }{h^2} + \frac{\|\Delta^h_t \boldf\|_{L^2(Q)}^2}{h^2}\Big).
\end{align*}
Using the regularity properties of \( (\phi_i)_{i=1}^\infty \) and the coefficients \( \betab^m \) of the finite-dimensional solution, we may use Lebesgue's dominated convergence theorem when taking the limit as \( h\rightarrow 0+ \) to deduce that
\begin{equation}\label{pp:equ34}
\begin{aligned}
&\sup_{t\in [0, T]}\|\ddot{\bu}^m(t) \|_2^2 + \int_Q \frac{|\dot{\bbT}^m|^2}{(1 + |\bbT^m|)^{1+a}} \,\mathrm{d}x\,\mathrm{d}t
\\
&\quad\quad
\leq C(a,\alpha, T) \Big( \|\ddot{\bu}^m(0) \|_2^2 + \|\dot{\bu}^m (0) \|_2^2 + \|\dot{\boldf}\|_{L^2(Q)}^2 \Big)
\\
&\quad\quad
\leq C(a,\alpha, T) \Big( \|\ddot{\bu}^m(0) \|_2^2 + \|\bv_0 \|_2^2 + \|\dot{\boldf}\|_{L^2(Q)}^2 \Big).
\end{aligned}
\end{equation}
To bound \( (\ddot{\bu}^m(0) )_m \) in \( L^2(\Omega)^d\), independent of \( m\), we note that by the structure of the finite-dimensional problem  the following holds:
\begin{equation}\label{pp:equ32}
\begin{aligned}
\|\ddot{\bu}^m(0) \|_2^2 &= \int_\Omega \ddot{\bu}^m(0) \cdot \ddot{\bu}^m(0) \,\mathrm{d}x
\\
&= \int_\Omega - \bbT^m(0): \beps(\ddot{\bu}^m(0) ) + \boldf(0) \cdot \ddot{\bu}^m(0) \,\mathrm{d}x
\\
&= \int_\Omega \Big(\mathrm{div}( F_n^{-1}(\beps(P^m(\bv_0 + \alpha\bu_0)))) + \boldf(0)\Big) \cdot \ddot{\bu}^m(0) \,\mathrm{d}x
\\
&\leq \Big(\|\mathrm{div}( F_n^{-1}(\beps(P^m(\bv_0 + \alpha\bu_0))))\|_2 + \|\boldf(0) \|_{2}\Big)\|\ddot{\bu}^m(0) \|_2.
\end{aligned}
\end{equation}
To bound the first term on the right-hand side of (\ref{pp:equ32}),  we first note that
\begin{equation}\label{pp:equ33}
\begin{aligned}
&\|\mathrm{div}( F_n^{-1}(\beps(P^m(\bv_0 + \alpha\bu_0))))\|_2
\\&\quad\quad\leq \|D(F_n^{-1}) (\beps(P^m(\bv_0 + \alpha\bu_0 ))) \|_\infty\|D(\beps(P^m (\bv_0 + \alpha\bu_0 )))\|_{2}
\\&\quad\quad
\leq
\|D(F_n^{-1}) (\beps(P^m(\bv_0 + \alpha\bu_0 ))) \|_\infty\|\bv_0 + \alpha\bu_0 \|_{2,2},
\end{aligned}
\end{equation}
where \( F_n^{-1}\) is continuously differentiable by Lemma \ref{pp:lem6}. By properties of symmetric, positive definite matrices, we are also able to show that there exists a constant \( C_a \) depending only on the parameter \( a\) such that
\begin{equation}\label{pp:equ48}
|D(F_n^{-1})(\bbS)| = |(DF_n)^{-1}(F_n^{-1}(\bbS))| \leq C_a ( 1 + |F_n^{-1}(\bbS)|^{a+1}),
\end{equation}
for every \( \bbS\in\mathbb{R}^{d\times d}\).
The first equality comes from use of the inverse function theorem.
Using (\ref{pp:equ48}) and (\ref{pp:equ33}) in (\ref{pp:equ32}) yields
\begin{align*}
\|\ddot{\bu}^m(0) \|_2 &\leq \|\boldf(0) \|_2 + \|(DF_n)^{-1}(\bbT^m(0))\|_\infty \|\bv_0 + \alpha\bu_0 \|_{2,2}
\\
&\leq \|\boldf(0) \|_2 + C_a( 1 + f^{-1}(C_1)^{a+1}) \|\bv_0 + \alpha\bu_0 \|_{2,2}
\\
&\leq \|\boldf\|_{L^\infty(0, T; L^2_*(\Omega))}+ C_a( 1 + f^{-1}(C_1)^{a+1}) \|\bv_0 + \alpha\bu_0 \|_{2,2},
\end{align*}
for every \( m \geq m_0 \), where \( m_0 \) and \( C_1 \) are the constants from the proof of Theorem \ref{pp:thm1}. Substituting this into (\ref{pp:equ34}), we deduce that
\begin{align*}
&\sup_{t\in [0, T]}\|\ddot{\bu}^m(t) \|_2^2 + \int_Q \frac{|\dot{\bbT}^m|^2 }{(1 + |\bbT^m|)^{1 + a}} \,\mathrm{d}x\,\mathrm{d}t
\\&\quad\quad
\leq C(a,\alpha, T) \Big( \|\boldf\|_{L^\infty(0, T; L^2_*(\Omega))}+ ( 1 + f^{-1}(C_1)^{a+1}) \|\bv_0 + \alpha\bu_0 \|_{2,2} + \|\bv_0 \|_2^2 + \|\dot{\boldf}\|_{L^2(Q)}^2\Big)
\\
&\quad\quad
\leq C(a, \alpha, T, \bu_0, \bv_0, \boldf),
\end{align*}
for every \( m\geq m_0 \), where \( C\) is a constant that is independent of \( m \) and \( n \). Thus (\ref{pp:equ28}) is satisfied and the proof is complete.
\end{proof}

\section{Existence of a solution to the strain-limiting problem}\label{pp:seclimitn}
\begin{theorem}\label{pp:thm2}
Let $\alpha>0$ and $a>0$. Assume that \( \bu_0 \), \( \bv_0 \in W^{1,2}_*(\Omega)^d\) are given such that \( \bv_0 + \alpha\bu_0 \in W^{k+1,2}_*(\Omega)^d\) for some \( k > \frac{d}{2}\) with
\[
\|\beps(\bv_0 + \alpha\bu_0 )\|_\infty \leq C_* < 1,
\]
for a constant \( C_* \in (0, 1) \). Let \(  \boldf\) be an element of \( W^{1,2}(0, T+\epsilon; L^2_*(\Omega)^d) \cap  L^2 (0, T; W^{1,2}_*(\Omega)^d)\), for an \( \epsilon>0 \).  Then, there exists a unique weak solution \( (\bu, \bbT) \) of the strain-limiting problem (\ref{pp:equ1}) in the sense of Definition \ref{pp:def1}.

Furthermore, if \( (\bu^n, \bbT^n) \) denotes the weak solution of (\ref{pp:equ2}) for \( n \in \mathbb{N}\), the following convergence results hold:
\begin{itemize}
\item \( \bu^n \overset{\ast}{\rightharpoonup}\bu\) weakly-* in \( L^\infty(0, T; W^{1,2}_*(\Omega)^d) \);
\item \( \dot{\bu}^n \overset{\ast}{\rightharpoonup}\dot{\bu}\) weakly-* in \( L^\infty(0, T; W^{1,2}_*(\Omega)^d) \);
\item \( \ddot{\bu}^n \overset{\ast}{\rightharpoonup} \ddot{\bu}\) weakly-* in \(L^\infty(0, T; L^2_*(\Omega)^d) \);
\item \( \dot{\bu}^n + \alpha\bu^n \rightharpoonup \dot{\bu}+ \alpha\bu \) weakly in \( L^p(0, T; W^{1,p}_*(\Omega)^d) \) for every \( p \in [1,\infty) \);
\item \( \bbT^n \rightarrow\bbT \) pointwise a.e. in \( Q\).
\end{itemize}
\end{theorem}

\begin{proof}
The proof of uniqueness is almost identical to that of Theorem \ref{pp:thm1} so we only prove existence here.
First, we would like to show that \( (\bbT^n)_n \) converges pointwise a.e. on \( Q\). To this end, we use an argument similar to one in \cite{RN8} but adapted to the time-dependent problem under consideration here.
We  let \( (\bu^{n,m}, \bbT^{n,m}) \) denote the solution to the Galerkin approximation from \( V_m \)  of the regularised problem with the parameter \( n \), and \((\bu^n, \bbT^n)\) the weak solution of (\ref{pp:equ2}). We define sequences \( (\bbS^{n,m})_m\) and \( (s^{n,m})_m\) by
\begin{align*}
\bbS^{n,m} = \frac{\bbT^{n,m}}{(1 + |\bbT^{n,m}|)^{a+1}} \quad \text{ and }\quad s^{n,m} = \frac{1}{( 1+ |\bbT^{n,m}|)^{a+1}}.
\end{align*}
Trivially, we have
\[
\|\bbS^{n,m}\|_{L^\infty(Q)} + \|s^{n,m}\|_{L^\infty (Q)} \leq 2.
\]
From direct computation, we get
\[
|\dot{\bbS}^{n,m}| + |\dot{s}^{n,m}| \leq C\frac{|\dot{\bbT}^{n,m}|}{(1 + |\bbT^{n,m}|)^{\frac{1}{2} + \frac{a}{2}}},
\]
for a constant \( C\) depending only on \( a\). An analogous result holds for the spatial derivatives. Using (\ref{pp:equ28}) and (\ref{pp:equ29}), it follows that
\begin{align*}
\int_Q |\dot{\bbS}^{n,m}|^2 + |\dot{s}^{n,m}|^2 + |\nabla \bbS^{n,m}|^2 + |\nabla s^{n,m}|^2 \,\mathrm{d}x\,\mathrm{d}t \leq C,
\end{align*}
for every \( m\geq m_0 \), where \( C\) is a constant that is independent of \( m \) and \( n \). By weak compactness, for each fixed \( n \), the sequence \( (\bbS^{n,m})_m \) converges weakly in \( W^{1,2}(0, T; L^2_\#(\Omega)^{d\times d}) \) and \( L^2(0, T; W^{1,2}_\#(\Omega)^{d\times d}) \).  However, by the pointwise convergence of \( (\bbT^{n,m})_m \) to \(\bbT^n\), it follows that the limit of \((\bbS^{n,m})_m \) in the above spaces is \( \bbS^n = \frac{\bbT^n}{(1 + |\bbT^n|)^{a+1}}\). An analogous result holds for \((s^{n,m})_m \) and \( s^n = (1 + |\bbT^n|)^{-a-1}\).
Furthermore, by weak lower semi-continuity of the norm we have
\begin{align*}
\int_Q |\bbS^n|^2 + |\dot{\bbS}^n |^2 + |\nabla \bbS^n|^2 + |s^n|^2 + |\dot{s}^n|^2 + |\nabla s^n|^2 \,\mathrm{d}x\,\mathrm{d}t\leq C,
\end{align*}
where \( C\) is independent of \( n\). Applying the Aubin--Lions lemma, we deduce that \( (\bbS^n)_n \) and \( (s^n)_n \) converge strongly in \( L^2(0, T; L^2_\#(\Omega)^{d\times d}) \) and  \( L^2(0, T; L^2_\#(\Omega)) \), respectively. Thus, up to a subsequence that we do not relabel, the sequences converge pointwise a.e. on \( Q\). Let us call the limits \( \overline{\bbS}\) and \( \overline{s}\), respectively.

However, by Fatou's lemma and (\ref{pp:equ27}), we know that \( \overline{s}^{-1-a}\) is an element of \( L^1(Q)^{d\times d}\). Thus \( \overline{s} > 0 \) a.e. on \( Q\) and we deduce that \( \bbT^n = S^n(s^n)^{-1}\) converges pointwise a.e. on \( Q\) as \( n\rightarrow\infty \). We denote the limit \( \bbT\). By Fatou's lemma, we  have \( \bbT \in L^1(Q)^{d\times d}\).

Thanks to the bounds stated in (\ref{pp:equ27}), (\ref{pp:equ28}) and (\ref{pp:equ29}), we deduce that the following convergence results hold:
\begin{itemize}
\item \( \bu^n \overset{\ast}{\rightharpoonup} \bu\) weakly-* in \( L^\infty(0, T; W^{1,2}_*(\Omega)^d) \);
\item \( \dot{\bu}^n \overset{\ast}{\rightharpoonup} \dot{\bu}\) weakly-* in \( L^\infty(0, T; W^{1,2}_*(\Omega)^d) \);
\item \( \ddot{\bu}^n \overset{\ast}{\rightharpoonup} \ddot{\bu}\) weakly-* in \( L^\infty(0, T; L^2_*(\Omega)^d) \);
\item \( \dot{\bu}^n + \alpha\bu^n {\rightharpoonup} \dot{\bu}+\alpha\bu \) weakly in \( L^p(0, T; W^{1,p}_*(\Omega)^d) \) for every \( p \in [1,\infty) \).
\end{itemize}
The attainment of the initial data (\ref{pp:lem5}) follows immediately by reasoning in much the same way as we did in the proof of Theorem \ref{pp:thm1}. As a matter of fact, we have the following stronger result:
\begin{align*}
\lim_{t\rightarrow 0+} \Big( \|\bu(t) - \bu_0 \|_{1,2} + \|\dot{\bu}(t) - \bv_0 \|_{1,2}\Big) = 0.
\end{align*}

For (\ref{pp:equ4}), we note that \( (F(\bbT^n))_n \) converges pointwise a.e. on \( Q\) to \( F(\bbT)\). Also \( (\beps(\dot{\bu}^n + \alpha\bu^n))_n \) converges weakly in \( L^2(Q)^{d\times d}\) to \( \beps(\dot{\bu}  +\alpha\bu) \). Since pointwise and  weak limits in \( L^2(Q)\) must coincide, (\ref{pp:equ4}) is immediate if we can show that \( (\frac{\bbT^n}{n(1 + |\bbT^n|^{1-\frac{1}{n}})})_n \) converges pointwise a.e. in \( Q\) to \( \mathbf{0}\). However, we have
\begin{align*}
\int_Q \bigg|\frac{\bbT^n}{n(1 + |\bbT^n|^{1-\frac{1}{n}})}\bigg|^2 \,\mathrm{d}x\,\mathrm{d}t
&\leq \int_Q \frac{1}{n^2} + \frac{|\bbT^n|^\frac{2}{n}}{n^2} \chi_{\{ |\bbT^n|\geq 1\}}\,\mathrm{d}x\,\mathrm{d}t
\\
&\leq \frac{|Q|}{n^2 } + \frac{1}{n}\int_Q \frac{|\bbT^n|^{1 + \frac{1}{n}}}{n}\,\mathrm{d}x\,\mathrm{d}t
\\
&
\leq \frac{C}{n},
\end{align*}
where \( C\) is a constant that is independent of \( n \). Hence the sequence converges strongly in \( L^2(Q)^{d\times d}\) (and thus pointwise a.e. in $Q$, up to subsequence) to \( \mathbf{0}\) as \( n\rightarrow\infty \) as required.

It remains to show that (\ref{pp:equ3}) holds.
We use similar reasoning to that in \cite{RN6} but we repeat all the details for completeness.
Let \( \tau \in C^1_c(\mathbb{R}) \), \( \bv\in C^1_*(\overline{\Omega})^d\) and \( \chi \in C([0, T]) \). Since \( \ddot{\bu}^n(t) \) and \( \boldf(t) \) are elements of \( L^2_*(\Omega)^d\), we can use any  \( \bv\in W^{n+1}_\#(\Omega)^d \) as a test function in (\ref{pp:equ21}), rather than only those with integral over \( \Omega\) equal to \( 0 \). Thus \( \tau(|\bbT^n|) \bv \chi \) is a valid test function in (\ref{pp:equ21}). We integrate the result over \( (0, T)\) to get
\begin{align*}
0 & = \int_Q \ddot{\bu}^n \cdot \chi \bv \tau(|\bbT^n|) + \bbT^n:\beps(\chi \bv\tau(|\bbT^n|)) - \boldf \cdot (\chi \bv \tau (|\bbT^n|) \,\mathrm{d}x\,\mathrm{d}t
\\
&= \int_Q\ddot{\bu}^n \cdot \chi \bv \tau(|\bbT^n|) +\tau(|\bbT^n|)\chi \bbT^n :\beps(\bv) - \boldf \cdot (\chi \bv \tau (|\bbT^n|)
+ \chi \bbT^n :\nabla \tau(|\bbT^n|)\otimes \bv\,\mathrm{d}x\,\mathrm{d}t.
\end{align*}
Using the compactness of the  support of \( \tau \) and the pointwise convergence of \( (\bbT^n)_n \), we know that \( \tau(|\bbT^n|) \rightarrow\tau(|\bbT|) \) and \( \tau(|\bbT^n|)\bbT^n \rightarrow \tau(|\bbT|) \bbT\)  strongly in \( L^p(Q)\) for every \( p \in [1,\infty) \). Thus we have
\begin{align*}
&\int_Q \ddot{\bu}\cdot \chi \bv \tau( |\bbT|) + \chi\tau(|\bbT|)\bbT:\beps(\bv) - \chi \tau(|\bbT|)\boldf \cdot \bv\,\mathrm{d}x\,\mathrm{d}t
\\
&\quad\quad
=\lim_{n\rightarrow\infty}\int_Q\ddot{\bu}^n \cdot \chi \bv \tau(|\bbT^n|) +\tau(|\bbT^n|)\chi \bbT^n :\beps(\bv) - \boldf \cdot (\chi \bv \tau (|\bbT^n|)\,\mathrm{d}x\,\mathrm{d}t
\\
&\quad\quad
= -\lim_{n\rightarrow\infty} \int_Q \chi \bbT^n :\nabla \tau(|\bbT^n|) \otimes \bv\,\mathrm{d}x\,\mathrm{d}t.
\end{align*}
We replace \( \tau \) by \( \tau_k \in C^1_c(\mathbb{R}) \) where \( \tau_k \) is such that
\begin{align*}
\tau_k(s) = \begin{cases}
1, &\quad |s|\leq k,
\\
0, &\quad |s|\geq 2k,
\end{cases}
\end{align*}
and \( |\tau_k^\prime(s) |\leq \frac{C}{k}\) for a constant \( C\) that is independent of \( k \). Letting \( k\rightarrow\infty \) and using Lebesgue's dominated convergence theorem, we deduce that
\begin{equation}\label{pp:equ35}
\int_Q \ddot{\bu}\cdot \chi \bv + \chi \bbT:\beps(\bv) - \chi \boldf \cdot \bv \,\mathrm{d}x\,\mathrm{d}t = -\lim_{k\rightarrow\infty }\lim_{n\rightarrow\infty} \int_Q \chi \bbT^n :\nabla \tau_k(|\bbT^n|) \otimes \bv\,\mathrm{d}x\,\mathrm{d}t.
\end{equation}
The term inside the limit on the right-hand side can be rewritten as
\begin{align*}
\int_Q \chi \bbT^n :\nabla \tau_k(|\bbT^n|) \otimes \bv\,\mathrm{d}x\,\mathrm{d}t
= \int_Q \chi b(|\bbT^n|) F(\bbT^n)_{ij} \frac{\partial}{\partial x_j}\Big(\tau_k(|\bbT^n|)\Big) v_i \,\mathrm{d}x\,\mathrm{d}t,
\end{align*}
where \( b\) is defined on \( [0, \infty ) \) by \( b(t) = (1 + t^a)^{\frac{1}{a}}\). Now define \( B_k \) on \( [0, \infty ) \) by
\[
B_k(t) = \int_0^t b(s) \tau_k^\prime(s) \,\mathrm{d}s.
\]
With this is mind, we get
\begin{align*}
&\int_Q \chi b(|\bbT^n|) F(\bbT^n)_{ij} \frac{\partial}{\partial x_j}\Big(\tau_k(|\bbT^n|)\Big) v_i \,\mathrm{d}x\,\mathrm{d}t
\\
&\quad\quad
= \int_Q \chi b(|\bbT^n|) \tau^\prime_k (|\bbT^n|) \frac{\partial}{\partial x_j}\Big( |\bbT^n|\Big) F(\bbT^n)_{ij}v_i\,\mathrm{d}x\,\mathrm{d}t
\\
&\quad\quad
= \int_Q \chi \frac{\partial}{\partial x_j}\Big( B_k(|\bbT^n|) \Big) F(\bbT^n)_{ij} v_i \,\mathrm{d}x\,\mathrm{d}t
\\
&\quad\quad
=
- \int_Q \chi B_k(|\bbT^n|) F(\bbT^n)_{ij}\frac{\partial v_i}{\partial x_j} + \chi B_k(|\bbT^n|) \frac{\partial}{\partial x_j}\Big( F(\bbT^n)_{ij}\Big) v_i \,\mathrm{d}x\,\mathrm{d}t
\\
&\quad\quad
=
- \int_Q \chi B_k(|\bbT^n|) F(\bbT^n)_{ij}\frac{\partial v_i}{\partial x_j} + \chi B_k(|\bbT^n|)\mathcal{A}_{ijpq}(\bbT^n) \frac{\partial}{\partial x_j}\Big( T^n_{pq}\Big) \,\mathrm{d}x\,\mathrm{d}t,
\end{align*}
where \( \mathcal{A}(\bbT) \) is the fourth-order tensor defined by
\[
\mathcal{A}_{ijpq}(\bbT) = \frac{\partial}{\partial T_{pq}} \left( \frac{T_{ij}}{(1 + |\bbT|^a)^{\frac{1}{a}}}\right) .
\]
We note that we can define an inner product on \( \mathbb{R}^{d\times d}\) by
\[
(\bbS, \mathbf{U})_{\mathcal{A}(\bbT) } = \sum_{i,j,p,q = 1}^d \mathcal{A}_{ijpq}(\bbT) S_{ij}U_{pq}.
\]
Using this definition, it follows that
\begin{align*}
&\Big|\int_Q \chi \bbT^n :\nabla \tau_k(|\bbT^n|) \otimes \bv\,\mathrm{d}x\,\mathrm{d}t\Big|
\\
&\quad\quad
\leq \int_Q |\chi B_k(|\bbT^n|) ||F(\bbT^n)||\nabla\bv| + \Big| \Big(\Big( \delta_{jl} B_k(|\bbT^n|) \frac{v_i }{|\bv|} \Big)_{i,j}, \chi |\bv|\partial_l \bbT^n \Big)_{\mathcal{A}(\bbT^n)}\Big|\,\mathrm{d}x\,\mathrm{d}t
\\
&\quad\quad
\leq \Big( \int_Q \Big(\Big(\delta_{jl} B_k(|\bbT^n|) \frac{v_i }{|\bv|} \Big)_{i,j}, \delta_{jl} B_k(|\bbT^n|) \frac{v_i }{|\bv|} \Big)_{i,j}\Big)_{\mathcal{A}(\bbT^n)}\,\mathrm{d}x\,\mathrm{d}t\Big)^{\frac{1}{2}}
\\
&\quad\quad\quad\quad
\cdot \Big( \int_Q (\chi|\bv|\partial_l \bbT^n, \chi|\bv|\partial_l \bbT^n)_{\mathcal{A}(\bbT^n)}\,\mathrm{d}x\,\mathrm{d}t\Big)^{\frac{1}{2}}
+ \int_Q |\chi B_k(|\bbT^n|)||F(\bbT^n)||\nabla \bv|\,\mathrm{d}x\,\mathrm{d}t.
\end{align*}
Thanks to the definition of \( \mathcal{A}(\bbT) \), we have
\begin{align*}
\int_Q (\chi |\bv|\partial_l \bbT^n, \chi |\bv|\partial_l \bbT^n)_{\mathcal{A}(\bbT^n)}\,\mathrm{d}x\,\mathrm{d}t
&\leq
\int_Q \chi^2 |\bv|^2 \nabla \bbT^n\,\vdots\, \nabla F (\bbT^n) \,\mathrm{d}x\,\mathrm{d}t
\\
&\leq
\int_Q \chi^2 |\bv|^2 \nabla \bbT^n\,\vdots\, \nabla F_n(\bbT^n) \,\mathrm{d}x\,\mathrm{d}t
\\
&\leq C(\chi, \bv) \int_Q \nabla \bbT^n\,\vdots \,\nabla F_n(\bbT^n) \,\mathrm{d}x\,\mathrm{d}t
\\
&\leq C,
\end{align*}
where \( C\) is a positive constant that is independent of \( n \).
Furthermore, we have
\begin{equation}\label{pp:equ36}
\begin{aligned}
&\int_Q \Big(\Big(\delta_{jl} B_k(|\bbT^n|) \frac{v_i }{|\bv|} \Big)_{i,j}, \delta_{jl} B_k(|\bbT^n|) \frac{v_i }{|\bv|} \Big)_{i,j}\Big)_{\mathcal{A}(\bbT^n)}\,\mathrm{d}x\,\mathrm{d}t
\\
&\quad\quad
\leq C(d) \int_Q \frac{|B_k(|\bbT^n|)|^2}{( 1 + |\bbT^n|^a)^{\frac{1}{a}}}\,\mathrm{d}x\,\mathrm{d}t,
\end{aligned}
\end{equation}
where \( C(d) \) is a positive constant depending only on \( d\).
By the choice of \( \tau_k \), we have \( B_k (t) = 0 \) if \( t\leq k \). If \( t\geq k\), we have
\begin{align}\label{pp:equ37}
\begin{aligned}
|B_k(t) | &= \Big| \int_k^t \tau^\prime_k(s)b(s) \,\mathrm{d}s \Big| \\
&\leq \frac{C}{k}\int_k^{2k} b(s) \,\mathrm{d}s \\
&\leq C(1 + k) \\
&\leq C(1 + t),
\end{aligned}
\end{align}
where \( C\) is independent of \( k \) and \( t\). Furthermore, \( B_k\) is uniformly bounded on \( [0,\infty) \). Using this and (\ref{pp:equ37}), taking the limit in (\ref{pp:equ36}) we get
\begin{align*}
&\lim_{k\rightarrow\infty}\lim_{n\rightarrow\infty} \int_Q \Big(\Big(\delta_{jl} B_k(|\bbT^n|) \frac{v_i }{|\bv|} \Big)_{i,j}, \delta_{jl} B_k(|\bbT^n|) \frac{v_i }{|\bv|} \Big)_{i,j}\Big)_{\mathcal{A}(\bbT^n)}\,\mathrm{d}x\,\mathrm{d}t
\\
&\quad\quad
\leq \lim_{k\rightarrow\infty}C\int_Q \frac{|B_k(|\bbT|)|^2}{(1 + |\bbT|^a)^{\frac{1}{a}}}\,\mathrm{d}x\,\mathrm{d}t
\\
&\quad\quad
\leq \lim_{k\rightarrow\infty} C\int_{\{|\bbT|> k\}} 1 + |\bbT|\,\mathrm{d}x\,\mathrm{d}t
\\
&\quad\quad
= 0,
\end{align*}
where the transition to the final line follows from the fact that \( \bbT \in L^1(Q)^{d\times d}\).
In particular, we deduce that
\begin{equation}\label{pp:equ38}
\lim_{k\rightarrow\infty}\lim_{n\rightarrow\infty}\int_Q \Big| \Big(\Big( \delta_{jl} B_k(|\bbT^n|) \frac{v_i }{|\bv|} \Big)_{i,j}, \chi |\bv|\partial_l \bbT^n \Big)_{\mathcal{A}(\bbT^n)}\Big|\,\mathrm{d}x\,\mathrm{d}t = 0.
\end{equation}
By similar reasoning, we have
\begin{equation}\label{pp:equ39}
\begin{aligned}
\lim_{k\rightarrow\infty}\lim_{n\rightarrow\infty} \int_Q |B_k(|\bbT^n|) ||F(\bbT^n)||\nabla \bv|\,\mathrm{d}x\,\mathrm{d}t
&= \lim_{k\rightarrow\infty}\int_Q |B_k(|\bbT|) ||F(\bbT)||\nabla \bv|\,\mathrm{d}x\,\mathrm{d}t
\\
&\leq \lim_{k\rightarrow\infty} C(\bv) \int_{\{ |\bbT|> k\}} 1 + |\bbT|\,\mathrm{d}x\,\mathrm{d}t
\\
&= 0.
\end{aligned}
\end{equation}
Combining (\ref{pp:equ38}) and (\ref{pp:equ39}), it follows that
\begin{align*}
\lim_{k\rightarrow\infty}\lim_{n\rightarrow\infty} \int_Q \chi\bbT^n:\nabla \tau_k(|\bbT^n|) \otimes \bv\,\mathrm{d}x\,\mathrm{d}t = 0.
\end{align*}
Returning to (\ref{pp:equ35}) and noting that \( \chi \) is arbitrary, we deduce that
\begin{equation}\label{pp:equ40}
\int_\Omega \ddot{\bu}(t) \cdot \bv + \bbT(t) : \beps(\bv) \,\mathrm{d}x = \int_\Omega \boldf \cdot \bv\,\mathrm{d}x,
\end{equation}
for a.e. \( t\in (0, T) \) and every \( \bv\in C^1_*(\overline{\Omega})^d\). Using Lemma \ref{pp:lem3}, it follows that (\ref{pp:equ40}) holds for every \( \bv\in W^{1,2}_*(\Omega)^d\) such that \( \beps(\bv) \in L^\infty_\#(\Omega)^d\). Hence \( (\bu,\bbT) \) is a weak solution of the strain-limiting problem  (\ref{pp:equ1}) and the proof is complete.
\end{proof}

Although we have a full existence result from Theorem \ref{pp:thm2}, under the condition that \( a\) is small and the dimension is \( 3\) we can improve the convergence result of Theorem \ref{pp:thm2} for \( (\bbT^n)_n\).

\begin{theorem}\label{pp:thm3}
Let \(\alpha>0 \), \(a \in (0, \frac{2}{7}) \) and \( d = 3\). Assume that \( \bu_0 \), \( \bv_0 \in W^{1,2}_*(\Omega)^d\) are such that \( \bv_0 + \alpha\bu_0 \in W^{k+1,2}_*(\Omega)^d\) for some \( k > \frac{d}{2}\) with
\begin{align*}
\|\beps(\bv_0 + \alpha\bu_0 )\|_\infty \leq C_* < 1.
\end{align*}
Suppose that \( \boldf\in  L^2(0, T; W^{1,2}_*(\Omega)^d) \).  Let \( (\bu^n, \bbT^n) \) be the unique weak solution of the regularised problem (\ref{pp:equ2}). Then there exists a couple \( (\bu, \bbT) \), the unique weak solution of  (\ref{pp:equ1}), such that
\[
\bbT^n \rightarrow\bbT \quad \text{ strongly in }L^1(0, T; L^1_\#(\Omega)^{d\times d}) .
\]
\end{theorem}

\begin{proof}
Under the conditions of the theorem, if \( (\bu^{n,m},\bbT^{n,m}) \) denotes the solution to the Galerkin approximation from \( V_m\), then we have
\begin{equation}\label{pp:equ41}
\sup_{t\in[0, T]}\Big( \int_\Omega |\bbT^{n,m}(t) |^{1-a}\,\mathrm{d}x \Big) + \int_Q \frac{|\nabla \bbT^{n,m}|^2}{( 1 + |\bbT^{n,m}|)^{a + 1}}\,\mathrm{d}x\,\mathrm{d}t \leq C,
\end{equation}
for every \( m \geq m_0 \), where \( C\) is independent of \( n \) and \( m \).
We no longer require the indicator function in the first term on the left-hand side because \( a < 1\).
From now on, we will assume that \( m\geq m_0 \).
Using the Sobolev embedding theorem, for every \( p \in (2,6 ]\) we have
\begin{equation}\label{pp:equ42}
\begin{aligned}
\int_0^T \Big( \int_\Omega|\bbT^{n,m}|^{\frac{p(1-a)}{2}}\,\mathrm{d}x\Big)^{\frac{2}{p}}\,\mathrm{d}t \leq C(p, a) \Big( \int_Q \frac{|\nabla \bbT^{n,m}|^2}{( 1 + |\bbT^{n,m}|)^{a + 1}}\,\mathrm{d}x\,\mathrm{d}t  + 1\Big).
\end{aligned}
\end{equation}
We note that we can choose \( p \) sufficiently large from  the interval \( (2,6]\) so that \( p(1-a)/2 \) is greater than \( 1\). We may take \( p \) larger if needed in subsequent calculations, but \( p\) will remain bounded by \( 6\).

Let \( q\in (1,\infty) \), to be determined later. Let \( q^\prime = \frac{q}{q-1}\) denote the H\"{o}lder conjugate of \( q\). Using H\"{o}lder's inequality, for a.e. \( t\in (0, T) \), we have
\begin{equation}\label{pp:equ60}
\begin{aligned}
\int_\Omega |\bbT^{n,m}|^{ 1+ \frac{a}{q^\prime}}\,\mathrm{d}x
&\leq \Big( \int_\Omega|\bbT^{n,m}|^{1-a}\,\mathrm{d}x\Big)^{\frac{1}{q^\prime}}\Big( \int_\Omega |\bbT^{n,m}|^{ 1 + 2a(q-1) }\,\mathrm{d}x \Big)^{\frac{1}{q}}
\\
&\leq C^{\frac{1}{q^\prime}} \Big( \int_\Omega |\bbT^{n,m}|^{ 1 + 2a(q-1)}\,\mathrm{d}x\Big)^{\frac{1}{q}},
\end{aligned}
\end{equation}
where the constant \( C\) comes from the first term on the left-hand side of (\ref{pp:equ41}). To bound the other factor on the right-hand side of (\ref{pp:equ60}), we would like to use (\ref{pp:equ42}). A pair of simultaneous restrictions on \( p \) and \( q\) that enable this are
\begin{equation}\label{pp:equ43}
\frac{1}{q}\leq \frac{2}{p},
\end{equation}
and
\begin{equation}\label{pp:equ44}
1 + 2a(q-1) \leq \frac{p(1-a)}{2}.
\end{equation}
To show that we may find  \( p \in (2,6]\) and \( q> 1\) such that the above holds when \( a\in (0, \frac{2}{7}]\), we first set \( p = 2q\). Then (\ref{pp:equ43}) automatically holds. We will now demand that \( q\in (1,3]\)  because  \( p \in (2,6]\).

We write \( a = \frac{1}{3}- \delta\) for a \( \delta\in (0, \frac{1}{3}) \). Then (\ref{pp:equ44}) reduces to
\[
\frac{1}{3\delta} + 2 \leq 3q.
\]
If \( \delta\geq \frac{1}{21}\), then \( \frac{1}{3\delta} + 2\leq 9\) and \( q\) can be chosen sufficiently large from the interval \( (1,3]\) so that
\[
\frac{1}{3\delta} + 2\leq 3q\leq 9.
\]
Thus (\ref{pp:equ44}) holds if \( \delta\geq  \frac{1}{21}\), which is simply equivalent to \( a\leq \frac{2}{7}\). Thus we have that
\begin{align*}
\int_Q |\bbT^{n,m}|^{1 + \frac{a}{q^\prime}}\,\mathrm{d}x\,\mathrm{d}t
&\leq C^{\frac{1}{q^\prime}}\int_0^T\Big( \int_\Omega |\bbT^{n,m}|^{1 + 2a(q-1)}\,\mathrm{d}x \Big)^{\frac{1}{q}}\,\mathrm{d}t
\\
&\leq C\Big[ \int_0^T \Big( \int_\Omega|\bbT^{n,m}|^{\frac{p(1-a)}{2}}\,\mathrm{d}x\Big)^{\frac{1}{q}}\,\mathrm{d}t + 1\Big]
\\
&\leq C\Big[\int_0^T \Big( \int_\Omega|\bbT^{n,m}|^{\frac{p(1-a)}{2}}\,\mathrm{d}x\Big)^{\frac{2}{p}}\,\mathrm{d}t + 1\Big]
\\
&\leq C(a,\alpha, \bu_0, \bv_0,\boldf,  \Omega, T),
\end{align*}
where \( C\) is a positive constant that is independent of \( n\) and \( m \). We note that the choices of \( p \) and \( q\) were made independent of \( n\) and \( m\). Fix \( \delta = \frac{a}{q^\prime} >0 \). Then, for a constant \( C\) that is independent of \(n\) and \( m\), we get
\[
\int_Q |\bbT^{n,m}|^{1 + \delta}\,\mathrm{d}x\,\mathrm{d}t \leq C.
\]
Using the convergence results in Theorems \ref{pp:thm1}, we deduce that
\[
\int_Q |\bbT^n|^{1 + \delta}\,\mathrm{d}x\,\mathrm{d}t.
\]
It follows that \( \bbT^n \rightharpoonup \bbT \) weakly in \(L^{1 + \delta}(Q)^{d\times d}\). Reasoning as we did in the proof of Theorem \ref{pp:thm1} and noting that \( (\frac{\bbT^n}{n(1 + |\bbT^n|^{1-\frac{1}{n}} ) })_n \) converges to \( \mathbf{0}\) strongly in \( L^2(Q)^{d\times d}\) as \( n\rightarrow\infty\), we deduce that \( \bbT^n\rightarrow\bbT \) strongly in \( L^1(0, T; L^1_\#(\Omega)^{d\times d}) \) as \(n\rightarrow\infty \). Using these two facts, we may repeat the reasoning from the proof of Theorem \ref{pp:thm1}  to deduce that \( (\bu, \bbT) \) is a weak solution of (\ref{pp:equ1}) under these weaker conditions.
\end{proof}

\begin{remark} \emph{
We remark that the results here hold under lower regularity requirements on the data. Indeed, rather than \( \bv_0 + \alpha\bu_0 \in W^{k+1,2}_*(\Omega)^d\) for some \( k > \frac{d}{2}\), we need only assume that \( \bv_0 + \alpha\bu_0 \in W^{2,2}_*(\Omega)^d\). However, in this case we need an extra level of approximation.
This is required to deal with the convergence of \( (\beps(P^m(\bv_0 + \alpha\bu_0)))_m \).
We consider the following regularised problem:
\begin{equation}\label{pp:equ45}
\begin{aligned}
\bu_{tt}&= \mathrm{div}(\bbT) + \boldf, &\quad &\text{ in }Q,
 \\
 \beps(\bu_t + \alpha\bu) &= \frac{\bbT}{(1 + |\bbT|^a)^{\frac{1}{a}}} + \frac{\bbT}{n} &\quad &\text{ in }Q,
 \\
 \bu(0, \cdot) &= \bu_0^k,  &\quad &\text{ in }\Omega,
 \\
 \bu_t(0, \cdot) &= \bv_0^k, &\quad &\text{ in }\Omega.
\end{aligned}
\end{equation}
The functions \( \bu_0^k\), \( \bv_0^k\) are chosen so that they are elements of \( C^\infty_*(\overline{\Omega})^d\) such that if \( \bu_0 \), \( \bv_0 \in L^2_*(\Omega)^d\) with \( \bv_0 + \alpha\bu_0 \in W^{2,2}_*(\Omega)^d\cap W^{1,\infty}_*(\Omega)^d\), then
\begin{itemize}
\item \( \bu_0^k \rightarrow\bu_0 \) strongly in \( L^2_*(\Omega)^d\);
\item \( \bv_0^k \rightarrow\bv_0 \) strongly in \( L^2_*(\Omega)^d\);
\item \( \bv_0^k + \alpha\bu_0^k \rightarrow\bv_0 + \alpha\bu_0 \) strongly in \( W^{2,2}_*(\Omega)^d\);
\item \( \beps(\bv_0^k + \alpha\bu_0^k) \overset{\ast}{\rightharpoonup}\beps(\bv_0 + \alpha\bu_0) \) weakly-* in \( L^\infty(\Omega)^{d\times d}\).
\end{itemize}
We deduce that such an approximating sequence exists by applying the proof of Lemma 4.1 in \cite{RN8} to \( \bu_0 \) and \( \bv_0 \). We deduce  the existence of a weak solution to (\ref{pp:equ45}) by using a finite-dimensional approximation as in Theorem \ref{pp:thm1}. Then we take the limit as \( k\rightarrow\infty \) for each fixed \( n\). Rather than considering \( f^{-1}(C_1) \) as we did in the proof of Theorem \ref{pp:thm1}, we use \( f_n^{-1}(\|\beps(\bv_0^k + \alpha\bu_0^k ) \|_\infty ) \) in the bounds on \( (\bu^{n,k}, \bbT^{n,k})\). This is uniformly bounded in \( k \), so we may deduce appropriate convergence results for each fixed \( n \). However, since we only have weak-* convergence of \( (\beps(\bv_0^k  + \alpha\bu_0^k))_k \), we must obtain bounds of the type (\ref{pp:equ19}) and (\ref{pp:equ29}) for \( (\bu^n, \bbT^n)\) from scratch. Since we have chosen the regularisation \(n^{-1}\bbT\), we have the following bound:
\begin{align*}
\int_Q \Big( |\dot{\bbT}^{n,k}|^2 + |\nabla \bbT^{n,k}|^2 + |\bbT^{n,k}|^2 \Big) \,\mathrm{d}x\,\mathrm{d}t \leq C(n),
\end{align*}
where \( C\) is a constant that is independent of \( k \) and \( (\bu^{n,k}, \bbT^{n,k}) \) is the weak solution of (\ref{pp:equ45}). Thus if \( (\bu^n,\bbT^n)\) is the solution of (\ref{pp:equ45}) but with \( \bu_0^k\), \( \bv_0^k \) replaced by \( \bu_0 \), \( \bv_0 \), then we must have \( \bbT^n \in C([0, T]; L^2(\Omega)^{d\times d}) \). Using this fact, we can derive estimates similarly as for the finite-dimensional solution in the proof of Theorem \ref{pp:thm1}. Since we are no longer using approximate initial data, we may use \( f^{-1}(C_*) \) in the bounds that are obtained.}\quad $\diamond$
\end{remark}

\begin{remark}
\emph{
We note that we may consider more general constitutive relations than the function \( F\). In particular, the above proofs can be adapted to a relation of the form}
\[
\beps(\bu_t + \alpha\bu) = f(|\bbT|) \bbT,
\] 
\emph{where \( f: [0,\infty) \rightarrow[0,\infty) \) is a continuously differentiable function and there exist positive constants \( C_1\), \( C_2\), \( \kappa\) and \( a\) such that}
\begin{equation}\label{pp:equ46}
\frac{C_1 s^2 }{\kappa + s} \leq f(s) s^2 \leq C_2 s
\end{equation}
\emph{and}
\begin{equation}\label{pp:equ47}
\frac{\mathrm{d}}{\mathrm{d}s}\Big( f(s) s\Big) \geq \frac{C_1}{(\kappa + s)^{a+1}},
\end{equation}
\emph{for every \( s\geq 0 \). Similarly, in the spirit of \cite{RN9}, we could also consider relations of the form
\begin{align*}
\beps(\bu_t + \alpha\bu) = \lambda(|\tr\bbT|)(\tr\bbT) \bbI + \mu(|\bbTd|) \bbTd,
\end{align*}
where \( \mu\) and \( \lambda\) satisfy conditions similar to (\ref{pp:equ46}) and (\ref{pp:equ47}).} \quad $\diamond$
\end{remark}

\section{Conclusion and open problems}\label{pp:secconclusion}
We have proved the existence of a unique global-in-time large-data weak solution to a class of initial-boundary-value problems that model the motion of implicitly constituted strain-limiting viscoelastic solids, without any restrictions on the parameters in the model or the number of space dimensions. Related studies have only been pursued to date in the quasi-static case or in one space dimension. The results presented here complement those for elastic solids in the static case discussed in \cite{RN8}.
Admittedly, our focus on the periodic setting here is a notable simplification. The extension of the results presented in this paper to problems on general domains, with Dirichlet and mixed Dirichlet--Neumann boundary conditions will be reported in a forthcoming paper.

An interesting further challenge is the proof of the existence of weak solutions to the purely elastic unsteady strain-limiting model. Although the analysis presented in this paper suggests that such an existence result might be within reach, it has to be borne in mind that the viscoelastic term has a regularising effect, whose absence in the purely elastic case is likely to lead to additional technical complications.

Finally, a related class of problems that has been discussed in the literature are stress-rate type models. For these the constitutive relation is of the form
\begin{equation}\label{pp:equ57}
\beps(\bu) = F(\bbT) - \gamma\bbT_t,
\end{equation}
where \( F\) is an appropriately chosen  function and \( \gamma > 0 \). In the case when \( F\) is bounded, it would be interesting to investigate the existence and uniqueness of solutions to (\ref{pp:equ1}), where the constitutive relation is now replaced by (\ref{pp:equ57}). Rudimentary  comparisons of solutions to the stress-rate problem and the strain-rate problem are performed in \cite{RN122}, but results concerning the existence of global weak solutions to multidimensional stress-rate type models are still lacking.

\section*{Acknowledgements}
M. Bul\'{\i}\v{c}ek's work is supported by the project 20-11027X financed by GA\v{C}R. M. Bul\'{\i}\v{c}ek is a member of the Ne\v{c}as Center for Mathematical Modeling. V. Patel is supported by the UK Engineering and Physical Sciences Research Council [EP/L015811/1]. Y. \c{S}eng\"{u}l is partially supported by the Scientific and Technological Research Council of Turkey (T\"{U}BITAK) under the grant 116F093.

\textsc{Email address:} \texttt{mbul8060@karlin.mff.cuni.cz}

\textsc{Email address:} \texttt{victoria.patel@maths.ox.ac.uk} 

\textsc{Email address:}  \texttt{yasemin.sengul@sabanciuniv.edu}

\textsc{Email address:} \texttt{endre.suli@maths.ox.ac.uk}

 \end{document}